\documentclass[a4paper, 11 pt, twoside]{amsart}
\usepackage[left = 3.5cm, right = 3.5cm, headsep = 6mm,
footskip = 10mm, top = 35mm, bottom = 35mm, footnotesep=5mm, headheight =
2cm]{geometry}
\usepackage{graphicx}
\usepackage[expansion=false]{microtype}
\usepackage{booktabs, multirow,  tabularx}

\usepackage{MnSymbol}
\usepackage[usenames, dvipsnames]{xcolor}

\usepackage{tikz, tikz-cd}
\usepackage{enumitem}
\newlist{prooflist}{description}{1}
\setlist[prooflist]{font=\normalfont \itshape, labelindent = \parindent, leftmargin = 0pt}
\newlist{proofenum}{enumerate}{1}
\setlist[proofenum, 1]{label = $(\roman{proofenumi})$,  wide}

\usepackage[unicode,bookmarks, pdftex, backref = page]{hyperref}
\hypersetup{colorlinks=true,citecolor=NavyBlue,linkcolor=NavyBlue,urlcolor=Orange, pdfpagemode=UseNone, breaklinks=true}

\newtheoremstyle{special-example}
  {}
  {}
  {}
  {\parindent}
  {\bfseries}
  {:}
  { }
  {}
  \theoremstyle{special-example}
\newtheorem{example}[equation]{Example}

\renewcommand{\tilde}{\widetilde}

\newcommand{\pp}{\mathbb P}

\newcommand{\OO}{\mathcal O}
\newcommand{\Diff}{\mathrm{Diff}}

\newcommand{\I}{\mathrm{i}}

\newcommand\rank{\operatorname{rank}}

\newcommand{\refl}[1]{^{[{#1}]}}

\newtheoremstyle{Lehn-it}
  {}
  {}
  {\itshape}
  {}
  {\bfseries}
  {$\;$\textmd{---}}
  { }
  {}

\newtheoremstyle{Lehn-up}
  {}
  {}
  {\upshape}
  {}
  {\bfseries}
  {$\;$\textmd{---}}
  { }
  {}

  \newtheoremstyle{up-list}
  {}
  {}
  {\upshape}
  {}
  {\bfseries}
  { }
  { }
  {}

\newtheoremstyle{Lehn-Bemerkung}
  {}
  {}
  {}
  {}
  {\itshape}
  {$\;$\textmd{---}}
  { }
  {}

\newtheoremstyle{citing}
  {}
  {}
  {\itshape}
  {}
  {\bfseries}
  {$\;$\textmd{---}}
  {.5em}
  {\thmnote{#3}}
  
\numberwithin{equation}{section}

{\theoremstyle{Lehn-it}

\newtheorem{thm}[equation]{Theorem}
\newtheorem{lem}[equation]{Lemma}
\newtheorem{prop}[equation]{Proposition}
\newtheorem{cor}[equation]{Corollary}

}
{\theoremstyle{Lehn-up}
\newtheorem{defin}[equation]{Definition}
\newtheorem{exam}[equation]{Example}

\newtheorem{construction}[equation]{Construction}
}
{\theoremstyle{Lehn-Bemerkung}
\newtheorem{rem}[equation]{Remark}
}

{\theoremstyle{up-list}

}

{\theoremstyle{citing}
}

\renewcommand{\qed}{\hspace*{\fill}$\square$\vspace{2ex}}

\newcommand{\onto}{\twoheadrightarrow}

\DeclareFontFamily{OT1}{rsfs}{}
\DeclareFontShape{OT1}{rsfs}{n}{it}{<-> rsfs10}{}
\DeclareMathAlphabet{\curly}{OT1}{rsfs}{n}{it}

\DeclareMathOperator{\Proj}{\mathrm{Proj}}

\DeclareMathOperator{\Pic}{Pic}

\newcommand{\shom}{\curly{H}om}

\newcommand{\shext}{\curly{E}xt}

\DeclareMathOperator{\Aut}{Aut}

\DeclareMathOperator{\rk}{rk}

\DeclareMathOperator{\tensor}{\otimes}

\newcommand{\isom}{\cong}

\newcommand{\restr}[1]{{\raisebox{-0.1\height}{$|_{#1}$}}}

\renewcommand{\epsilon}{\varepsilon}
\renewcommand{\phi}{\varphi}
\renewcommand{\theta}{\vartheta}

\newcommand{\kc}{{\mathcal C}}

\newcommand{\kf}{{\mathcal F}}

\newcommand{\ki}{{\mathcal I}}

\newcommand{\kl}{{\mathcal L}}

\newcommand{\ko}{{\mathcal O}}

\newcommand{\kw}{{\mathcal W}}
\newcommand{\kx}{{\mathcal X}}

\newcommand{\IA}{{\mathbb A}}

\newcommand{\IC}{{\mathbb C}}

\newcommand{\IF}{{\mathbb F}}

\newcommand{\IP}{{\mathbb P}}
\newcommand{\IQ}{{\mathbb Q}}

\newcommand{\IZ}{{\mathbb Z}}

\newcommand{\gothD}{{\mathfrak D}}

\newcommand{\gothM}{{\mathfrak M}}

\newcommand{\gothR}{{\mathfrak R}}

\newcommand{\gothU}{{\mathfrak U}}

\title{ 2-Gorenstein stable surfaces with $K_X^2 = 1$ and $\chi(X) = 3$}

\author[S. Coughlan]{Stephen Coughlan}
\address{Stephen Coughlan\\Department of Mathematics and Computer Studies\\ Mary Immaculate College\\South Circular Road\\Limerick\\ Ireland}
\email{stephen.coughlan@mic.ul.ie}

\author{Marco Franciosi}
\address{Marco Franciosi\\Dipartimento di Matematica\\Universit\`a di Pisa \\Largo B. Pontecorvo 5\\I-56127  Pisa\\Italy}
\email{marco.franciosi@unipi.it}

\author{Rita Pardini}
\address{Rita Pardini\\Dipartimento di Matematica\\Universit\`a di Pisa \\Largo B. Pontecorvo 5\\I-56127  Pisa\\Italy}
\email{rita.pardini@unipi.it}

\author{S\"onke Rollenske}
\address{S\"onke Rollenske\\FB 12/Mathematik und Informatik\\
Philipps-Universit\"at Marburg\\
Hans-Meerwein-Str. 6\\
35032 Marburg\\
Germany}
\email{rollenske@mathematik.uni-marburg.de}

\begin{document}

\begin{abstract}
The compactification $\overline \gothM_{1,3}$ of the Gieseker moduli space of surfaces of general type with $K_X^2 =1 $ and $\chi(X)=3$  in the moduli space of stable surfaces parametrises so-called stable I-surfaces.

We classify all such surfaces which are 2-Gorenstein into four types using a mix of algebraic and geometric techniques. We find a new divisor in the closure of the Gieseker component and a new irreducible component of the moduli space. 
\end{abstract}

\maketitle

\setcounter{tocdepth}{1}
\tableofcontents

\section{Introduction}

The Gieseker moduli  space $\gothM_{1,3}$ of canonical models of surfaces of general type with $K_X^2 =1$ and $\chi(X) = 3$ is a rational variety of dimension 28 and it was classically known that  the surfaces it parametrises are hypersurfaces of degree 10 contained in the smooth locus of $\IP(1,1,2,5)$. More geometrically speaking, the bicanonical map realises these surfaces as double covers of the quadric cone in $\IP^3$, branched over the vertex and a  sufficiently general quintic section. 
 
 Nowadays, the Gieseker  moduli space   $\gothM_{a,b}$  is known to admit a modular compactification $\overline{\gothM}_{a,b}$, the moduli space of stable surfaces with $K_X^2 =a$ and $\chi(X) = b$, sometimes called KSBA-moduli space after Koll\'ar, Shepherd-Barron, and Alexeev (compare \cite{KollarModuli}). For brevity we call the surfaces parametrised by $\overline{\gothM}_{1,3}$ stable I-surfaces (see Definition \ref{def: I surface} for a precise definition).

Our detailed understanding of the Gieseker moduli space, or classical component, and the small values of the invariants make  $\overline{\gothM}_{1,3}$ into a fertile testing ground to explore stable surfaces and their moduli.

So far a full picture seems out of reach, and the current approaches aim for classification under some extra conditions on the singularities. The first result in this direction was the extension of the classical description to Gorenstein stable I-surfaces \cite{FPR17a}, which was refined and explored further in \cite{CFPR22} from a Hodge theoretic perspective. In \cite{FPRR22, CFPRR23} we explored surfaces with few T-singularities, finding several divisors and an additional component. Meanwhile in \cite{GPSZ24} Gallardo,  Pearlstein, Schaffler  and Zhang found eight more divisors by considering the stable replacement of Gorenstein degenerations with an exceptional unimodal singularity; in \cite{RT24} it was shown that there are not more divisors of this kind.

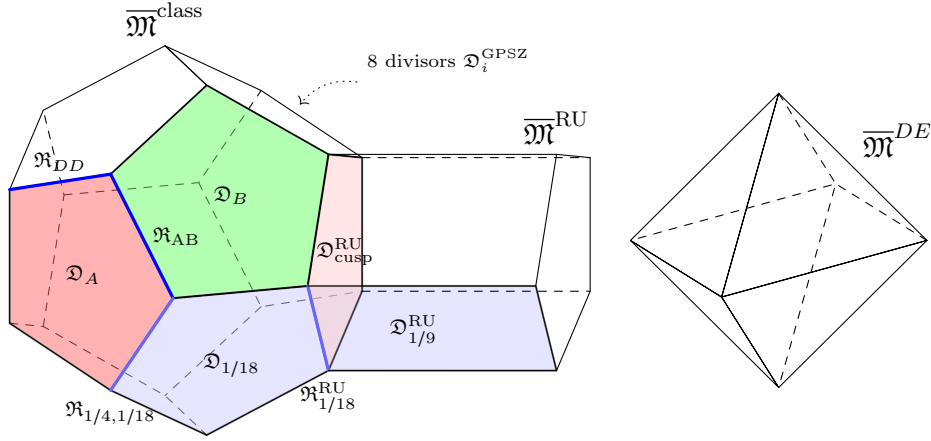
\begin{figure}[t]\label{fig: overview}
 \caption{Known strata in  $\overline{\gothM}_{1,3}$ (compare Table \ref{tab: strata} for notation)}
\begin{center}
\begin{tikzpicture}[scale = 1.5, line cap=round,line join=round]
\coordinate (P0) at (1.2508566957869456, 0.7603474449732492);
\coordinate (P1) at (1.2508566957869456, -1.1503255332779627);
\coordinate (P2) at (0.6598162824642664, 1.3249899183682845);
\coordinate (P3) at (0.6598162824642664, -0.5856830598829275);
\coordinate (P4) at (-0.6598162824642664, 0.5856830598829275);
\coordinate (P5) at (-0.6598162824642664, -1.3249899183682845);
\coordinate (P6) at (-1.2508566957869456, 1.1503255332779627);
\coordinate (P7) at (-1.2508566957869456, -0.7603474449732492);
\coordinate (P8) at (0.18264153207910092, 1.3712827900732547);
\coordinate (P9) at (0.18264153207910092, -1.7202510301231948);
\coordinate (P10) at (-0.18264153207910092, 1.7202510301231948);
\coordinate (P11) at (-0.18264153207910092, -1.3712827900732547);
\coordinate (P12) at (1.0685921597130592, -0.40283109341752804);
\coordinate (P13) at (0.11226868223217829, 0.5107796200274473);
\coordinate (P14) at (-0.11226868223217829, -0.5107796200274473);
\coordinate (P15) at (-1.0685921597130592, 0.40283109341752804);
\coordinate (P16) at (1.5457669100982248, 0.7317368768227392);
\coordinate (P17) at (-1.5457669100982248, 0.44912396512249836);
\coordinate (P18) at (1.5457669100982248, -0.44912396512249836);
\coordinate (P19) at (-1.5457669100982248, -0.7317368768227392);

   \path 
(P0) ++ (2,0) coordinate (Q0)
(P1) ++ (2,0) coordinate (Q1)
(P12) ++ (2,0) coordinate (Q12)
(P16) ++ (2,0) coordinate (Q16)
(P18) ++ (2,0) coordinate (Q18)
;

%
%
%
  
 \begin{scope}

  \begin{scope}[black]
    \draw
    (P0) -- (P8) -- (P10) -- (P2) -- (P16) -- cycle;
    \draw
    (P4) -- (P8) -- (P10) -- (P6) -- (P17) -- cycle;
    \draw
    (P0) -- (P8) -- (P4) -- (P14) -- (P12) -- cycle;
\draw 
(P16) -- (P0) -- (P12) -- (P1) -- (P18) -- cycle;
\draw
	(P1) -- (P12) -- (P14) -- (P5) -- (P9) -- cycle;
\draw 
	(P4) -- (P14) -- (P5) -- (P19) -- (P17) -- cycle;
	
      \end{scope}
  \begin{scope}[dashed]
  \draw(P19) -- (P7) -- (P11) -- (P9);
\draw(P7) -- (P15) -- (P6);
\draw(P15) -- (P13) -- (P2);
\draw(P13) -- (P3) -- (P18);
\draw(P3) -- (P11);

\node[above] at (P10) {$\overline\gothM^{\text{class}}$};

  \end{scope}
\begin{scope}[every node/.style = {font = \scriptsize}]

\filldraw[thick, fill = red!60, semitransparent] 
	(P4) -- (P14) -- (P5) -- (P19) -- (P17) -- cycle;
\node at (-.9,-.3 ) {$\gothD_A$};
	
\draw[thick, fill = green!60, semitransparent] 
 (P0) -- (P8) -- (P4) -- (P14) -- (P12) -- cycle;
\node at (.4,.4 ) {$\gothD_{B}$};

    \draw[thick, fill = blue!20, semitransparent] 
	(P1) -- (P12) -- (P14) -- (P5) -- (P9) -- cycle;
\node at (.4,-1.1 ) {$\gothD_{1/18}$};
  
    \draw[thick, fill = blue!20, semitransparent] 
	(P1) -- (Q1) -- (Q12) --  (P12) -- cycle;
\node at (2,-.8 ) {$\gothD^\text{RU}_{1/9}$};

\draw[thick, fill = red!20, semitransparent] 
(P16) -- (P0) -- (P12) -- (P1) -- (P18) -- cycle;
\node at (1.4,-.1 ) {$\gothD^\text{RU}_\text{cusp}$};

\draw[->, dotted] (1.5,1.4) node[above right]{\tiny 8 divisors $\gothD_i^\text{GPSZ}$} to[ bend right] ++ (-.5, -.2) ;

  
  \draw[very thick, blue]  (P17) to node[above, black] {$\gothR_{DD}$} (P4);
\draw[very thick, blue]  (P4) to node[right, black] {$\gothR_\text{AB}$} (P14);

    \draw[very thick, blue!60]  (P14) to  (P5) node[below, black] {$\gothR_{1/4, 1/18} $};
    \draw[very thick, blue!60]  (P12) to  (P1) node[below, black] {$\gothR^\text{RU}_{1/18} $};
    
 \end{scope}

 \end{scope}

  \begin{scope}

   \draw (P0) --  (Q0)node[above] {$\overline\gothM^{\text{RU}}$};
   \draw (P1) -- (Q1);
   \draw (P12) -- (Q12);
   \draw[dashed] (P16) -- (Q16);
   \draw[dashed] (P18) -- (Q18);
   
   \draw 
(Q16) -- (Q0) -- (Q12) -- (Q1) -- (Q18) -- cycle;
  \end{scope}

  \begin{scope}[scale = 1.3, line cap=round,line join=round, xshift = 4cm]

  \path
  ( 1, 0, 0) coordinate (A1)
  ( 0, 0,-1) coordinate (A2)
  (-1, 0, 0) coordinate (A3)
  ( 0, 0, 1) coordinate (A4)
  ( 0, 1, 0) coordinate (B1)
  ( 0,-1, 0) coordinate (B2);

  \begin{scope}[dashed]
    \draw
    (A1) -- (A2) -- (A3)
    (B1) -- (A2) -- (B2);
  \end{scope}

  \draw
  (A1) -- (A4) -- (B1)
  (A1) -- (A4) -- (B2)
  (A3) -- (A4) -- (B1)
  (A3) -- (A4) -- (B2)
  (B1) -- node[above right]  {$\overline \gothM^{DE}$} (A1) -- (B2) -- (A3) --cycle;

\end{scope}
\end{tikzpicture}

\end{center}

\end{figure}

In the present paper we completely  classify 2-Gorenstein stable I-surfaces, that is,  those with $2K_X$  Cartier, dividing  them into four types:
\begin{thm}\label{thm: four types}
Let $X$ be a 2-Gorenstein I-surface. Then $X$ is one of the following:
\begin{description}[font = \normalfont, nosep]
 \item[type A] Semi-log-canonical hypersurfaces in $\IP(1,1,2,5)$ not passing through the point $(0:0:0:1)$. This includes all Gorenstein stable I-surfaces and the divisor whose general element is a surface with one singularity of type $\frac 14(1,1)$. 
 \item[type B] These are reducible surfaces $X = X_1 \cup X_2$, where $X_1$ is a singular Enriques surface and $X_2$ is a singular K3 surface. They form a family of dimension $27$ in the closure of the Gieseker component, thus are smoothable, but no longer complete intersections.

 \item[type {DD}] These are reducible complete intersections of bidegree $(2,10)$ in weighted projective space $\IP(1,1,2,2,5)$ and the bicanonical map realises them as double covers of the union  of two planes. Each component is  a singular K3 surface and they form a subset of codimension two in the closure of the Gieseker component. 
 \item[type DE]  These are reducible surfaces $X = X_1 \cup X_2$, where both $X_i$ are singular K3 surfaces of a special kind. They form a 30-dimensional family and none of them is smoothable, that is, they form a new irreducible component. Their canonical ring is quite complicated.
\end{description}
We give an overview of the  known strata in $\overline{\gothM}_{1,3}$  
in Table \ref{tab: strata}, including dimension of the stratum, index of the general surface and a reference to precise information and proofs.
\end{thm}
Our current knowledge of the locus of 2-Gorenstein I-surfaces  in relation to the rest of $\overline{\gothM}_{1,3}$ is shown in Figure \ref{fig: overview}. 
\begin{table}
\caption{Known irreducible strata  in the moduli space $\overline \gothM_{1,3}$ of stable  I-surfaces}
\label{tab: strata}
  \begin{tabular}{ l ccp{6.1cm} l }
  \toprule
   Name & Dim. & index & closure of & reference \\
  \midrule
 $\overline{\gothM}^\text{class}$ &28 & 1 &  Gieseker component & \cite{BHPV, FPR17a}\\
 &&& & Theorem \ref{thm: surface-A}\\ 
 $\overline{\gothM}^\text{RU}$ &28 & 5 &  Rana--Urz\'ua component & \cite{FPRR22}\\
 $\overline{\gothM}^{DE}$ &  30& 2 & DE-component & Theorem \ref{thm: DE}\\
\midrule
 $\gothD_A$ & 27 & 2 & type A and not Gorenstein & Theorem \ref{thm: surface-A}\\ 
&&& resp.\ the $\frac 14(1,1)$ divisor &\cite{FPRR22} \\
 $\gothD_B$ & 27 & 2 & type B & Theorem \ref{thm: all type B}\\
 $\gothD_\text{cusp}^\text{RU}$ & 27 & 5 & cuspidal RU surfaces &\cite{CFPRR23}  \\
 $\gothD_{ 1/9}^\text{RU}$ & 27 & 15 & (nodal) RU surfaces with a $\frac 19 (  1,5)$ singularity  &\cite{CFPRR23} \\
 $\gothD_{ 1/{18}}$ & 27 & 3 & surfaces with a $\frac 1{18} (  1,5)$ singularity  &\cite{CFPRR23} \\
 $\gothD_i^\text{GPSZ} $ & 27 & $>2$ & 8 divisors from exceptional unimodal singularities &\cite{GPSZ24, RT24}\\
 \midrule
$\gothR_{{DD}}$ & 26& 2 & type {DD} & Theorem \ref{thm: DD} \\
$\gothR_{1/4, 1/18} $& 26  & 6 & surfaces with one $\frac 14 (1,1)$ and one $\frac 1{18} (  1,5)$ singularity & \cite{CFPRR23}\\
$\gothR_{1/18}^\text{RU} $& 26 & 15  & RU-surfaces with  one $\frac 1{18} (  1,5)$ singularity &
  \cite{CFPRR23} \\
  $\gothR_{AB}$ & 26& 2 & type B, in the closure of type A & Theorem \ref{thm: all type B} \\ 
   \bottomrule
  \end{tabular}

\end{table}

\subsection*{Outline of the paper and proof of Theorem \ref{thm: four types}}
Let $X$ be a stable 2-Gorenstein I-surface as defined in Definition \ref{def: I surface}. Then we show in Proposition \ref{prop: canonical section} that the general section of $\omega_X$ does not vanish on any component and thus defines a canonical curve $C$. 

If $C$ can be chosen reduced, then  we use the hyperplane section principle \cite{reid90} and our work on generalised Gorenstein spin structures on reduced curves of genus two \cite{CFPR23a} (summarised in Theorem \ref{thm: types and rings}) to compute the canonical ring. This gives the surfaces of type A and type B, classified in 
Sections \ref{sec: typeA} and \ref{sec: typeB}.

We were unable to argue along the same lines when the general canonical curve is non-reduced. However, in this case the surface is reducible and we succeed in giving  a geometric description of the components and the possible glueings, resulting in two cases  (Lemma \ref{lem: types non-reduced}).

Type {DD}  is identified in Section \ref{sect: DD} with a smoothable example already considered in \cite{FPR17a}. Type {DE} is shown to give a new irreducible component in Section \ref{sect: DE}; its  canonical ring is computed a posteriori from the geometric description.
\qed

\subsection*{Acknowledgements} 
S.R.\ is grateful for support by the DFG. 
 M.F.  and R.P. are  partially supported by the project PRIN 
 2022BTA242 ``Geometry of algebraic structures: Moduli, Invariants, Deformations''
  of Italian MUR  and members of GNSAGA of INDAM.  

\section{Notation and preliminary results}\label{notation}

\subsection{Set-up}
We work with schemes of finite type over the complex numbers.  

Given a sheaf $\kf$ on a scheme $X$, we denote by $\kf^{\vee}= \shom(\kf, \OO_X)$ its dual and 
 by $\kf^{[m]}$ the $m$-th reflexive power, i.e., the double dual of 
$\kf^{\tensor m}$.

A curve $C$  is a Cohen--Macaulay projective scheme of pure dimension 1 (possibly non-reduced or reducible).  It has a  dualising sheaf  $\omega_C$ and we denote by  $p_a(C)$ the arithmetic genus of $C$, that is, $p_a(C)=1-\chi(\OO_C)$. 
 The curve $C$ is Gorenstein if  $\omega_C$ is invertible; if this is the case  $K_C$ denotes a canonical divisor such that $\OO_C(K_C)\cong \omega_C$.

Our standard reference for stable surfaces is \cite{KollarSMMP}. 
A stable surface $X$ has semi-log-canonical singularities and ample canonical divisor. 
 In particular, it is by definition Cohen--Macaulay and Gorenstein 
in codimension one, so the canonical sheaf $\omega_X$ exists and is reflexive. We call $X$  $m$-Gorenstein if  the $m$-th reflexive power of the canonical sheaf $\omega_X^{[m]}$ is invertible.
 A canonical divisor is a Weil divisor $K_X$ such that  $\omega_X \cong \ko_X(K_X)$. 
 We use the notations $\chi(X) = \chi(\ko_X)$, $p_g(X) = h^0(X, \omega_X)$, and $q(X) = h^1(X, \ko_X)$.
 
 If $X$ is a non-normal stable surface, then   we denote by  $\pi\colon \bar X \to X$ the normalisation of $X$, and by $D\subset X $   and $ \bar D\subset \bar X$ the curves defined by the conductor ideal 
$ \shom_{\ko_X}(\pi_*\ko_{\bar X}, \ko_X)$. 
The invariants are related by   $K_X^2 = (K_{\bar X}+\bar D)^2$  and  $\chi(X) = \chi({\bar X})+\chi(D)-\chi({\bar D})$ (see also  \cite[Prop. 3.3]{FPR15a}).

The map $\pi\colon \bar D \to D$ on the conductor divisors is generically 
a double cover and thus  induces a rational  involution  on $\bar D$. Normalising the conductor loci we get an honest involution $\tau\colon \bar D^\nu\to \bar D^\nu$ such that $D^\nu = \bar D^\nu/\tau$ 
and such that the different $\Diff_{\bar D^\nu}(0)$ is $\tau$-invariant, which fits in the 
 the following pushout diagram:
\begin{equation}\label{diagr: pushout}
\begin{tikzcd}
    \bar X \dar{\pi}\rar[hookleftarrow]{\bar\iota} & \bar D\dar{\pi} & \bar D^\nu \lar[swap]{\bar\nu}\dar{/\tau}
    \\
X\rar[hookleftarrow]{\iota} &D &D^\nu\lar[swap]{\nu}
    \end{tikzcd}.
\end{equation}
By Koll\'ar's glueing construction \cite[Thm.~5.13]{KollarSMMP}, the surface $X$ can be reconstructed uniquely from the triple $(\bar X, \bar D, \tau)$.

 \subsection{I-surfaces}\label{section: I-surf}
Our main interest stems from  the Gieseker moduli space $\gothM_{1,3}$ and its closure in the moduli space of stable surfaces. Let us give  a precise definition. 
\begin{defin}\label{def: I surface}
 A stable I-surface is a stable surface with $K_X^2=1$, $\chi(X) = 3$ and Hilbert series of the canonical ring $\frac{1-t^{10}}{(1-t)^2(1-t^2)(1-t^5)}$, in particular $p_g(X) = 2$ and $q(X)=0$.
 
 We denote the moduli space of stable I-surfaces by $\overline\gothM_{1,3}$, the closure of the Gieseker component inside it by $\overline\gothM_{1,3}^\text{class}$.
 \end{defin}
 For canonical or just Gorenstein stable surfaces $X$ with fixed $K_X^2$ and $\chi(X)$ the Riemann--Roch formula and Kodaira vanishing determine all higher plurigenera as $P_m(X) = \chi(X) + \frac 12 m(m-1)$ for $m\geq2$. Together with the geometric genus  $p_g$ or equivalently the irregularity $q$ this determines the Hilbert function of the canonical ring $R(X, K_X)$. 
 
By \cite{FPR17a} and Proposition \ref{prop: canonical section} below a 2-Gorenstein stable surface with $K_X^2=1$ and $\chi(X) = 3$ satisfies $q(X) = 0$ and is an I-surface, but the assumption on the Hilbert series excludes some  components containing surfaces of higher index. 
 \begin{exam}
  Let $X=X_9\subset \IP(1,1,3,3)$ a general hypersurface of degree $9$. Then $X$ has three singularities of type $\frac 13 (1,1)$ and $\omega_X = \ko_X(1)$. Thus $3K_X$ is Cartier and ample, so $X$ is stable. The invariants are  $K_X^2 = 1$, $\chi(X) = 3$ with $p_g(X) = 2$. However, $P_2(X) = H^0(X, \ko_X(2)) = 3\neq \chi(X) +K_X^2$.
  
  This example and more general instances of this phenomenon were considered in \cite{rollenske23}.
 \end{exam}

 \section{Existence of canonical curves and consequences}
 \label{sec: canonical curves}
 If  $X$ is  a stable surface, we say that $X$ contains a canonical curve if there is a section $x_0\in H^0(X, \omega_X)$ which is non-zero at each generic point.
 In the irreducible  case this just means that  $x_0\neq0$, in general it can happen that a section vanishes identically on some but not every component. 
 \begin{example} 
 Let $C$ be a smooth plane quartic, $X_1=S^2C$ and $D_1$ a coordinate curve. We have $p_g(X_1) = q(X_1) = 3$ and $K_{X_1}^2 = 6$ (see    \cite[Example 1]{HP02}). Since $D_1$ is ample, the Riemann--Roch formula gives $h^0(K+D_1)=3=h^0(K)$, so that  $D_1$ is in the fixed part of $|K+D_1|$. 
 
 Now take $(X_2, D_2)=(\pp^2, C)$ and glue in the obvious way to get a stable surface  $X$, which consists of the two components intersecting with normal crossings in the curve $C$.  One can check that $K_X^2 = 13$, $p_g(X) = 3$, and $q(X) = 0$.

 By the above, all canonical sections vanish on the intersection curve, because their pullback to the normalisation vanishes on $D_1$, thus they vanish identically on the component $X_2$ and there is no canonical curve. 
  \end{example}

 \subsection{Existence of canonical curves}
 
In this section we prove the existence of canonical curves on any   $2$-Gorenstein I-surface. We start with slightly weaker hypotheses, which we then  prove to imply that we have an I-surface.

\begin{lem}\label{lem: X-reducible}
 Let $X$ be a reducible 2-Gorenstein  stable surface with $K_X^2 = 1$ and $\chi(X) = 3$. 
 Then:
\begin{enumerate} 
\item the normalization $(\bar X,\bar D)$ of $\bar X$ is equal to $(\bar X_1, \bar D_1)\sqcup (\bar X_2, \bar D_2)$,  with $\bar X_i$ irreducible and  $\bar D_i>0$, $i=1,2$; 
\item for $i=1,2$ the divisor $K_{\bar X_i}+\bar D_i$ is 2-Cartier and ample with $(K_{\bar X_i}+\bar D_i )^2=\frac 12$.
\end{enumerate}
\end{lem}
\begin{proof}
Let $\pi\colon \bar X = \bigsqcup_{i=1}^r\bar X_i \to X$ be the normalisation.
Since $2K_X$ is an ample Cartier divisor, its pullback $\pi^* 2K_X|_{\bar X_i} = 2( K_{\bar X_i} + \bar D_i)$ is an ample Cartier divisor as well, and \[ 2 = 2K_X^2 = \sum_{i=1}^r \left(2(K_{\bar X_i} + \bar D_i)\right)(K_{\bar X_i} + \bar D_i) .\]
Because the intersection of an ample Cartier divisor with a Weil divisor is a positive integer and because by assumption we have more than one component, we get $r = 2$ and $(K_{\bar X_i} + \bar D_i)^2= \frac 12$. 
The divisors $ \bar D_i$  cannot be zero, because they contain the preimage of the intersection of the two components, which is a  curve because $X$ is $S_2$ and connected.
\end{proof}

If $X$ is a reducible 2-Gorenstein stable surface with $K_X^2 = 1$ and $\chi(X) = 3$ and  normalization $(\bar X_1, \bar D_1)\sqcup (\bar X_2, \bar D_2)$, we write $\bar D_i= \bar Z_i+ \bar \Gamma_i$, 
where $\bar \Gamma_1$ is glued to  $ \bar \Gamma_2$ by $\pi$, while $ \bar Z_i$ is glued to itself for $i=1,2$. Note that $ \bar \Gamma_i>0$ since $X$ is connected and $S_2$, while $\bar Z_i$ may be zero. 

With this notation in place we can state the main result of this section:

\begin{prop}\label{prop: canonical section}
Let $X$ be a 2-Gorenstein stable surface with $K_X^2 = 1$ and $\chi(X) = 3$. Then:
\begin{enumerate}
\item  the zero locus of a general section of $H^0(K_X)$ is 1-dimensional;
\item $q(X)=0$ and $X$ is an I-surface;
\item all  canonical curves are non-reduced if and only if  $X=X_1\cup X_2$ is reducible and $K_{\bar X_i}+\bar Z_i=0$ for $i=1,2$. In this case, every  canonical curve is  supported on $X_1\cap X_2$.
\end{enumerate}
\end{prop}

Before  proving Proposition \ref{prop: canonical section} we give some auxiliary results. 
\begin{lem}\label{lem: aux}
Let $Y$ be a smooth projective surface and let $M$ be a nef  line bundle. 
If  $M^2=2$, then $h^0(M)\le 4$. 
\end{lem}
\begin{proof}
We can of course assume that $r:=h^0(M)-1\ge 1$. 
Write $|M|=Z+|D|$,  where $Z$ is the fixed part and $D$ the moving part,  denote by $h\colon X\to  \pp^r$ the map defined by $|D|$  and let $d$ be the degree of   the image $\Sigma$  of $h$.   Assume first that $\Sigma$ is a surface: then since $M$ and $D$ are  nef we have $2=M^2\ge MD\ge D^2\ge d\ge r-1$, namely $r=3$.
If instead $\Sigma$ is a curve, then $D$ is numerically equivalent to $dG$, where $G$ is irreducible and $2=M^2\ge MD=dMG\ge d\ge r$, so $r\le 2$ in this case.
\end{proof}

If  $X=X_1\cup X_2$ is a reducible stable 2-Gorenstein stable surface with $K_X^2 = 1$ and $\chi(X) = 3$, then we denote by 
\[\rho=\rho_1\oplus  \rho_2\colon H^0(K_X)\to H^0(K_{\bar X_1}+\bar D_1)\oplus H^0(K_{\bar X_2}+\bar D_2) \]the pullback map. 
\begin{lem}\label{lem: rho}
In the above set-up and notation, assume that $\rho_1$ is not injective: then $\dim \ker \rho_1=1$ and $K_{\bar X_2}+ \bar Z_2=0$.
\end{lem} 
\begin{proof}
Set $L_i:=2(K_{\bar X_i}+\bar D_i)$ for $i=1,2$. Recall that by assumption $L_i$ is an ample  line bundle with $L_i ^2=2$. 
Consider  $0\ne \sigma\in \ker \rho_1$: since $\rho$ is injective, $\rho_2(\sigma)$ is a non-zero section of $K_{\bar X_2}+\bar D_2$ that vanishes on $\bar \Gamma_2$, hence $K_{\bar X_2}+ \bar Z_2\ge 0$. Since $L_2$ is ample,  $1=L_2(K_{\bar X_2}+\bar D_2)=L_2(K_{\bar X_2}+ \bar Z_2)+L _2 \bar \Gamma_2$ 
and $\bar \Gamma_2>0$ give $L \bar \Gamma_2=1$, $L_2(K_{\bar X_2}+ \bar Z_2)=0$, therefore  $K_{\bar X_2}+ \bar Z_2=0$ and $ \bar \Gamma_2$ is irreducible.
 Note that $\bar \Gamma_2$ is the divisor of $\rho_2(\sigma)$. 
If $0\ne \tau$ is another element of $\ker \rho_1$, then the divisor of $\rho_2(\tau)$ is also equal to $\bar \Gamma_2$ and therefore, $\rho_2(\tau)$ and $\rho_2(\sigma)$ are linearly dependent. 
Since $\rho$ is injective, it follows that $\ker \rho_1$ has dimension 1. 
\end{proof}

\begin{proof}[Proof of Proposition \ref{prop: canonical section}]
\begin{proofenum}
 \item
Assume for contradiction that all sections of $K_X$ have a  2-dimensional zero locus. 
Then $X=X_1\cup X_2$ is reducible and we may assume that all sections of $K_X$ vanish on  $X_1$. By Lemma \ref{lem:  rho} we get $p_g(X)\le 1$, contradicting $\chi(X)=3$. 
 \item 
 Assume for contradiction $q(X)>0$, namely $p_g(X)\ge 3$. If $X$ is irreducible, then  $4=K^2_X+\chi(X)=h^0(2K_X)$ by    \cite[Prop. 16]{liu-rollenske14}, because $2K_X$ is Cartier. On the other hand the Hopf lemma (see \cite[p. 108]{ACGH}) gives $h^0(2K_X)\ge 2p_g(X)-1\ge 5$, a contradiction. 
 So $X=X_1\cup X_2$ is reducible. 
 
  If $\rho_i$ is injective, then $h^0(K_{\bar X_i}+\bar D_i)\ge 3$. Then the  Hopf lemma gives $h^0(2(K_{\bar X_i}+\bar D_i))\ge 2h^0(K_{\bar X_i}+\bar D_i)-1\ge 5$. Pulling back $2(K_{\bar X_i}+\bar D_i)$ to a line bundle $M$  on a desingularization $Y$ of $\bar X_i$ we obtain a contradiction to Lemma \ref{lem: aux}. So $\rho_1$ and $\rho_2$ are not injective. Identifying $H^0(K_X)$ with its image via $\rho$, we can find three independent sections of the form: $(\sigma_1,0)$, $(0,\tau_1)$, $(\sigma_2,\tau_2)$. Note that $\sigma_1$ and $\sigma_2$ are independent because $\ker \rho_2$ is $1$-dimensional by Lemma \ref{lem: rho} and by the same argument $\tau_1$ and $\tau_2$ are also independent. Now the following correspond to linearly independent sections of $2K_X$:
$$(\sigma_1^2,0), \ (\sigma_1\sigma_2,0),\  (0,\tau_1^2),\  (0,\tau_1\tau_2), \ (\sigma_2^2, \tau_2^2),$$
contradicting again $h^0(2K_X)=4$. We have proved that $q(X) = 0$. 

It remains to control the plurigenera of $X$. Since even multiples of the canonical divisor are Cartier, the Riemann--Roch formula applies, so we only need to control the odd plurigenera.

Let $C$ be a canonical curve defined by a section $x_0$. The reflexive restriction sequence \eqref{seq 1} below defines a torsion free sheaf of rank one $\kl = \omega_X|_C\refl{1}$ on $C$ with $\chi(\kl) = \chi(\ko_X) + \chi(\omega_X) = 0$. By definition we have  $\deg \kl = 1$, see \cite{hartshorne86, CFHR}.
We now twist \eqref{seq 1} with the line bundle  $\omega_X\refl{2m}$ and compute using generalised Kodaira vanishing and  the Riemann--Roch formula on singular curves
\begin{align*}
P_{2m+1} (X)& = \chi(\omega_X\refl{2m+1} )
= \chi(\omega_X\refl{2m})
 + \chi(\kl\tensor
\omega_X^{[2m]}
 |_C)
 \\
& = \chi (X) + m(2m-1) + \chi(C) + 2m+1 = \chi (X) + m(2m+1).
\end{align*}
This is the correct plurigenus for an I-surface.

 \item
  If $X$ is irreducible, then pulling back the canonical system to a desingularisation $Y$ of $X$ and considering its moving part  we see that the general canonical curve $C$ has at least a reduced component. Since $2K_XC=2$ and $2K_X$ is an ample line bundle, it follows that $C$ is reduced. Therefore $X=X_1\cup X_2$. Consider the pull-back map $\rho=\rho_1\oplus \rho_2$: if, say, $\rho_1$ is injective, then a similar argument shows that the restriction to $X_1$ of a general $C\in |K_X|$ is reduced and not contained in $\bar \Gamma_1$, hence $C$ is reduced. So it follows that $\rho_1$ and $\rho_2$ are not injective and, by Lemma \ref{lem: rho} $K_{\bar X_i}+ \bar Z_i=0$, $i=1,2$, and all canonical curves are supported on $X_1\cap X_2$. 

Conversely, assume that $X= X_1\cup X_2$ is reducible and $K_{\bar X_i}+ \bar Z_i=0$. Then $K_{\bar X_i}+\bar D_i= \bar \Gamma_i>0$ for $i=1,2$. Take $\sigma_i\in H^0(K_{\bar X_i}+\bar D_i)$  a non-zero section that vanishes on $\bar \Gamma_i$:  then $(\sigma_1,0)$ and $(0,\sigma_2)$ correspond to independent canonical sections. Since $p_g(X)=2$ by $(ii)$, these sections  generate $H^0(K_X)$ and therefore every canonical curve is supported on $X_1\cap X_2$ and is not reduced.
\qedhere
 \end{proofenum}
\end{proof}
 
\subsection{Restriction of the canonical ring to a canonical curve}
In order to apply Reid's hyperplane section principle, we need to describe the restriction of the canonical ring to a canonical curve carefully. In the end, this will be useful  only if the general canonical curve is reduced.

Let $X$ be a stable I-surface and fix $C$ a canonical curve on $X$ defined by $x_0 \in H^0(X, \omega_X)$. Noting that $\ko_X(C) = \omega_X$ is the canonical bundle we get three exact sequences
\begin{gather}
\begin{tikzcd}[ampersand replacement = \&] 0 \rar \&  \ko_X(-C) \rar{\cdot x_0} \&  \ko_X \rar \&  \ko_C \rar \&  0 \end{tikzcd},\label{seq 0}\\
\begin{tikzcd}[ampersand replacement = \&] 0 \rar \&  \ko_X \rar{\cdot x_0} \&  \omega_X \rar \&  \kl \rar \&  0 \end{tikzcd},\label{seq 1}\\
 \begin{tikzcd}[ampersand replacement = \&] 0 \rar \&  \omega_X \rar{\cdot x_0}  \&  \omega_X (C) = \omega_X\refl{2} \rar \&  \omega_C  \rar \&  0  \end{tikzcd},\label{seq 2}
\end{gather}
where we use that \eqref{seq 2} arises also by applying $\shom_{\ko_X}(-, \omega_X)$ to the restriction sequence \eqref{seq 0}, therefore identifying the third sheaf in \eqref{seq 2} as $\omega_C = \shext^1_{\ko_X}(\ko_C, \omega_X)$. 
\begin{prop}\label{prop: extra info I surfaces}
 Let $X$ be a 2-Gorenstein stable I-surface with canonical curve $C$ defined by $x_0 \in H^0(X, \omega_X)$. Then
 \begin{enumerate}
  \item The curve $C$ is a Gorenstein curve of arithmetic genus $2$ with ample canonical bundle $\omega_C$ and $H^0(C, \ko_C) = 1$.
  \item The sheaf $\kl$ is a torsion-free sheaf with $\chi(\kl) = 0$, $h^0(\kl) =1$ and the map $\mu$ defined by the diagram
  \[ 
   \begin{tikzcd}
    \omega_X \tensor_{\ko_X} \omega_X \rar \dar & \omega_X\refl{2}\dar\\
    \kl\tensor_{\ko_C} \kl \rar{\mu} & \omega_C
   \end{tikzcd}
  \]
is an isomorphism in codimension zero.
\item The multiplication on global sections induced by $\mu$, namely 
\[ H^0( \kl\tensor \omega_C^{\tensor m}) \times H^0( \kl\tensor \omega_C^{\tensor n}) \to H^0( \kl\tensor \kl \tensor \omega_{C}^{\tensor(m+n)}) \to H^0( \omega_C^{\tensor(m+n+1)}),
\]
makes the natural restriction map 
\[\phi\colon R(X, K_X) \onto \bigoplus_{n \geq 0} \left(H^0(C, \omega_C^{\otimes n}) \oplus H^0(C, \kl\otimes  \omega_C^{\otimes n})\right)=: R(C, \{ \kl, \omega_C\})\]
into a surjective  ring homomorphism with kernel generated by $x_0$. 
 \end{enumerate}
\end{prop}
In the notation of \cite{CFPR23a} the pair $(\kl, \mu)$ is a ggs (generalised Gorenstein spin) structure on $C$ and 
 $R(C, \{ \kl, \omega_C\})$ is the associated half-canonical ring.
 
 \begin{proof}
By the depth-Lemma \cite[Cor. 18.6]{Eisenbud}   and the fact that $\omega_X\refl{m}$ is $S_2$ for every $m$, the sheaf $\kl_i$ is the quotient of $\omega_X\tensor_{\ko_X}\ko_C$ by the torsion submodule. 

As $\omega_X\refl{2} $ is locally free, its restriction to $C$ has no torsion, so by the same argument $\omega_C$ is locally free and of degree $2 = 2K_X^2$. 

The map $\mu$ exists because every torsion sheaf maps to zero in the line bundle $\omega_C$. Outside the codimension two subset of $X$  where $\omega_X$ is not locally free, that is, outside a finite number of points, the map $\mu$ is an isomorphism. 

Now the existence of the ring structure on $R(C, \{\kl, \omega_C\}) $ and the ring homomorphism $\phi$ is clear by construction. 

All claims related to the dimension of cohomology groups on $C$, respectively the surjectivity of $\phi$, follow from the long exact cohomology sequences of \eqref{seq 0}, \eqref{seq 1}, \eqref{seq 2}, possibly twisted with $\omega_X\refl{2k}$,  using $q(X) = 0 $ by Proposition \ref{prop: canonical section} and generalised Kodaira vanishing \cite[Prop. 21]{liu-rollenske14}). 
 \end{proof}

 \begin{rem}
  It is only slightly more tedious to work out the nature of the restriction of a section ring of any $\IQ$-Cartier divisor to a curve in a similar fashion, but we don't need this here. 
 \end{rem}

 In the case where the canonical curve is reduced, we can use the results of  \cite{CFPR23a} to describe the restricted canonical ring $R(C, \{\kl, \omega_C\})$.
 
 \begin{thm} \label{thm: types and rings}
 Let $(C, \kl, \mu)$ be as in Proposition \ref{prop: extra info I surfaces}, that is, $C$ is a Gorenstein curve of arithmetic genus two with ample canonical bundle and $(\kl, \mu)$ is a ggs structure on $C$ with $h^0(C, \kl) = 1$. If $C$ is reduced, then the following two cases are possible
  \begin{description}[font = \normalfont, nosep]
  \item[type A] 
The curve   $C$ is a flat double cover of $\IP^1$, so in particular  it is either  irreducible or the union of  two smooth rational curves. The half-canonical ring is
\[ R(C, \{ \kl, \omega_C\}) = \IC[x,y,z]/( z^2 - f_{10}(x,y)),\] 
where $\deg(x,y,z) =(1,2,5)$ and $f_{10}\neq 0$. 
    \item[type B] 
The curve   $C$ is the union of two irreducible curves $C_i$  with $p_a(C_i)=1$, $i=1,2$, that meet transversely at a single  point $p$ that is smooth for both.
The half-canonical ring is
\[R(C, \{ \kl, \omega_C\}) = \IC[x,y,w,v,z,u] / I, \]
with   $\deg(x,y,w,v,z,v) =(1,2,3,4,5,6)$ and the ideal $I$ is  generated by the following  equations
 \[
 \rk \begin{pmatrix}0&y&w&z\\x&w&v&u\end{pmatrix}\leq 1, \begin{array}{rcl}
z^2 & = & yg_8(y,v) \\
zu & = & wg_8(y,v) \\
u^2 & = & vg_8(y,v) + x^4h_8(x,v)
\end{array}  
\]
  where $v^2$ appears in $g_8$ with non-zero coefficient. 
  \end{description}
 \end{thm}
\begin{proof}
 The distinction and description of the cases comes from Proposition 3.3 and Corollary 3.9 in \cite{CFPR23a}, while the half-canonical rings are the cases $A(1)$ and $B(1)$ of \cite[Theorem 5.2]{CFPR23a}.
\end{proof}

\section{Surfaces with reduced canonical curve of type A}\label{sec: typeA}

It turns out that these are well known to us. 
\begin{thm}\label{thm: surface-A}
 Let $X$ be a $2$-Gorenstein $I$-surface containing a  reduced canonical curve $C$  of  type $A$ and let 
 \[ \gothD_A = \overline{ \left\{ [X]\in \overline\gothM_{1,3} \left| \,  \text{\begin{minipage}[c]{7cm}
 $X$ is 2-Gorenstein but not Gorenstein with reduced canonical curve of type A       
                                                                          \end{minipage}
 }
 \right.\right\} }\]
 where the closure is taken in $\overline\gothM_{1,3}$.
 \begin{enumerate}
 \item The surface $X$ is canonically embedded as a hypersurface of degree $10$ in $\IP(1,1,2,5)$
  not passing through $(0:0:0:1)$ and  $X$ is $\IQ$-Gorenstein smoothable. It is Gorenstein if and only if $X$ does not contain the point $(0:0:1:0)$.

  Conversely, any such hypersurface with slc singularities is a stable I-surface.
  
  \item The set $\gothD_A$ is an irreducible divisor in the closure of the Gieseker component. 
  It coincides  with the closure of the divisor of surfaces with one singularity of type $\frac{1}{4}(1,1)$   considered in \cite{FPRR22, CFPRR23}. 
    \end{enumerate}
    \end{thm}
\begin{proof}
\begin{proofenum}
\item  The description as a hypersurface follows  directly via  the 
 hyperplane section principle \cite{reid90} from the description of the ring $R(C, \{ \kl, \omega_C\})$  given in Theorem  \ref{thm: types and rings} $(i)$. The ``converse'' statement follows by using adjunction for hypersurfaces in  $\IP(1,1,2,5)$, while the criterion for $X$ being Gorenstein was proved in  \cite[Theorem 3.3, Proposition 4.1]{FPR17a} (see also \cite[Remark  3.2]{CFPRR23}). 
 
 \item The set  $\gothD_A$ is the closure of the image of an open subset of the linear system of hypersurfaces of degree $10$ containing the point $(0:0:1:0)$ by $(i)$ and as such an irreducible divisor. Its general element has a unique singular point of type $\frac 14(1,1)$ by the discussion in \cite[Section 3.A]{FPRR22}.\qedhere
 \end{proofenum}
 \end{proof}

\begin{example}
If we  take $X$ to be the hypersurface in $\IP(1_{x_1},1_{x_2},2_y,5_z)$ defined by the equation  $z^2-x_1x_2f_4(x_1,x_2,y)^2=0$, with $f_4$ a general polynomial of degree 4, then $X$ is irreducible with normalization $\pp(1,1,4)$, and the canonical system has only one base  point, occurring above the vertex of the quadric cone.  Nevertheless, all the canonical curves are reducible. 
\end{example}

\section{Surfaces with reduced canonical curve of type B}\label{sec: typeB}
In this section we describe 2-Gorenstein I-surface with reduced canonical curve of type B (in the notation of Theorem \ref{thm: types and rings}). 
\begin{thm} \label{thm: all type B}
Let $X$ be a $2$-Gorenstein I-surface with reduced canonical curve of type B. Then the canonical ring of $X$ is as in Proposition \ref{prop: algebraic type B} and if $X$ is general, it is glued from an Enriques surface and a K3 surface as in Proposition \ref{prop: geometry B}.

The closure of the set of such surfaces, 
\[ \gothD_B = \overline{ \left\{ [X]\in \overline\gothM_{1,3} \left| \,  \text{\begin{minipage}[c]{5cm}
 $X$ 2-Gorenstein with reduced canonical curve of type B       
\end{minipage}
 }
 \right.\right\} } \subset \overline{\gothM}_{1,3}.\]
 is an irreducible divisor in the closure of the Gieseker moduli space.

 The intersection of divisors  $\gothD_A \cap \gothD_B$ contains the irreducible codimension two  subset 
 \[\gothR_{AB}  = \{ [X] \in \gothD_B \mid \text{ $\tilde g_8$ does not contain $y^4$ in \eqref{eq: B2}}\}.\]
 \end{thm}

 We will start with the geometric description, via which the surfaces were initially found, and then prove Theorem 
 \ref{thm: all type B} in Proposition \ref{prop: algebraic type B}, Corollary \ref{cor: B decomposition} and Proposition \ref{prop: locating divisor B}.
 
\subsection{A geometric construction}\label{sect: B geometry}
We will start with  a geometric description of 2-Gorenstein stable surfaces with reduced canonical curve of type B. Let us fix some notation: we denote by $\IF_n$ the Hirzebruch surface with negative section $s_\infty $ of square $s_\infty^2 = -n$. We call a  section $s_0$ disjoint from $s_\infty $ a positive section.

\begin{construction}\label{constr: X1}
 Let $Y_1$ be an Enriques surface containing a half-pencil $E_1$, that is, an elliptic curve $E_1$ such that $|2E_1|$ is an elliptic pencil. Assume further that $Y_1$ contains a $(-2)$-curve which is  a bisection $\Lambda_1$ of $|2E_1|$. 

 Contracting $\Lambda_1$ to an $A_1$ point $p$, we get a singular Enriques surface $X_1$ containing an elliptic curve $E_1$ through the point $p$ with $E_1^2 = \frac 12$.

 The K3 cover $T_1 \to Y_1$ is an elliptic K3 surface and  contains two disjoint $(-2)$-curves $G_1, G_2$, which realise $T_1$ as a double cover of $\IF_2$ branched over a (symmetric) divisor $B_1\in |4s_0|$ and one can construct explicit examples in this way. 
 
Such Enriques surfaces have been studied in \cite[Ch.4, \S7]{Dolgachev-Enriques} where it is  shown that the so-called superelliptic linear system  $|4E_1+2\Lambda_1|$ on $Y_1$ realises $X_1$ as the double cover of a degenerate symmetric del Pezzo surface $W_1'$ of degree 4 in $\pp^4$. 

One can  check that $W_1'\isom W_1 = \{ w^2-yv =0 \} \subset \pp(1_{x_0},2_y, 3_w,4_v)$ fitting in the following commutative diagram:
\begin{equation}\label{Enriques diagram} \begin{tikzcd}[column sep = large]
 && \IF_2\dar\\ 
    T_1 
    \arrow[out = 45, in = 180]{urr}{2:1 \text{ branched}}[swap]{\text{over } B_1\in |4s_0|}
    \rar{\text{contract}}[swap]{G_1, G_2}\dar{/\iota}[swap]{\text{K3 cover}} & S_1 \rar{2:1}\dar{/\iota} & \pp(1,1,2)\dar{/\iota} \\
     Y_1 \rar{\text{contract}}[swap]{\Lambda_1} \arrow[out = -45, in = 180]{drr}[swap]{|4E_1+2\Lambda_1|} & X_1 \rar{2:1} & W_1 \rar[hookrightarrow]\dar{\isom} & \IP(1,2,3,4)\dar{|\ko(4)|}\\
     &&W_1'\rar[hookrightarrow] &\IP^4
    \end{tikzcd}.
\end{equation}

Finally, we count the number of moduli of this construction.
The K3 cover $S_1$ of $X_1$ is a hypersurface of degree $8$ in $\pp(1_{a_0},1_{a_1},2_b,4_c)$  with  equation 
$c^2=g_8(a_0,a_1^2,a_1b,b^2)$. In addition, $g_8$ is invariant under the involution $\iota$ 
of $\pp(1,1,2,4)$ defined by $(a_0,a_1,b,c)\mapsto(a_0,-a_1,-b,-c)$ (cf. \cite[Thm.~VIII.18.2]{BHPV}). An explicit computation shows that $g_8$ varies in a linear system of dimension 12. Since the subgroup of automorphisms of $\pp(1,1,2,4)$ that commute with $\iota$ has dimension 3, we get $9$ parameters for $(X_1, E_1)$. 
\end{construction}

\begin{construction}\label{constr: X2}
Let $B_2
\subset \IF_4$ be a general divisor in $|s_\infty + 3s_0| = s_\infty + |3s_0|$ and let $Y_2$ be the double cover branched over $B_2$. By the Hurwitz formula $Y_2$ is a smooth K3 surface with an elliptic fibration induced by the fibration on $\IF_4$. 

The preimage of $s_\infty$ is a $(-2)$-curve $\Lambda_2\subset Y_2$ and contracting it we get  a singular K3 surface $X_2$ with one $A_1$ singularity at a point $p$. It can be viewed as a double cover of $\IP(1,1,4)$ or as a hypersurface of degree $12$ in $\IP(1,1,4,6)$, resulting in a diagram
\[
 \begin{tikzcd}[column sep = large]
  Y_2 \rar{\text{contract}}[swap]{\Lambda_2}\dar{2:1} & X_2 \dar{2:1} \rar[hookrightarrow] & \IP(1,1,4,6)\arrow[dashed]{dl}\\
  \IF_4 \rar {\text{contract}}[swap]{s_\infty}& \IP(1,1,4).
 \end{tikzcd}
\]
The image of any fibre of the elliptic fibration is a curve $E_2$ of arithmetic genus one in $X_2$ containing the point $p$ and satisfying $E_2^2 = \frac 12$. Note that a general $Y_2$ is not isotrivial, so every  isomorphism class of elliptic curves occurs a finite number of times as a fibre.

To count parameters for $Y_2$, we observe that the linear system $|\sigma_{\infty}+3\sigma_0|$ on $\IF_4$ has dimension 27, while the automorphism group of $\IF_4$ has dimension $9$, so we get $18$ parameters.
\end{construction}
\begin{prop}\label{prop: geometry B}
Let $(X_1, E_1)$ be as in Construction 
\ref{constr: X1} and $(X_2, E_2)$ be as in Construction \ref{constr: X2} with $E_1\isom E_2$ a smooth elliptic curve. Then the two pairs can be glued to a 2-Gorenstein stable  I-surface $X = X_1\cup_{E_1\isom E_2}X_2$ such that the general canonical curve is of type B.

The construction depends on $27$ parameters.
\end{prop}
\begin{proof}
Note that $(X_i, E_i)$ are stable pairs with volume $\frac 12$.
By Koll\'ar's glueing construction \cite[Thm.~5.13]{KollarSMMP}, the two components can be glued to a stable surface if the glueing isomorphism $E_1\isom E_2$ preserves the different, that is, if it preserves the point $p$. 

It is easy to check from the classification of slc singularities that $X$ is $2$-Gorenstein. The two components of a general canonical curve $C$ are defined by sections in $H^0(X_i, K_{X_i}+E_i)$ and are thus two curves of arithmetic genus one, intersecting in the point $p$. 

The pair $(X_1, E_1)$ depends on $9$ parameters (see Construction \ref{constr: X1}) and $X_2$ depends on $18$ parameters (see Construction \ref{constr: X2}), which leaves us with finitely many choices of a fibre $E_2$ of the elliptic fibration on $Y_2$ which is isomorphic to $E_1$.
\end{proof}
\begin{rem}
The canonical system of a surface as in Proposition \ref{prop: geometry B} has a fixed part: on the Enriques surface $X_1$ there is a unique section of $K_{X_1}+E_1$, corresponding to the other half-pencil in the elliptic fibration on $Y_1$. 
\end{rem}

\subsection{Classification via the canonical ring}\label{sect: B algebraic}
To completely classify 2-Gorenstein stable I-surfaces with reduced canonical curve of type B, we again rely on a description of the canonical ring and show that this coincides with the geometric description given in Section \ref{sect: B geometry}.

\begin{prop}\label{prop: algebraic type B}
 Let $X$ be a $2$-Gorenstein $I$-surface with reduced canonical curve of  type $B$. Then the surface 
  $X$ is canonically embedded in $\pp(1_{x_0},1_x,2_y,3_w,4_v,5_z,6_u)$  with weighted homogeneous equations:
\begin{equation}\label{eq: B2}
\bigwedge^2\begin{pmatrix}0&y&w&z\\x&w&v&u\end{pmatrix}=0,\ \
\begin{array}{rcl}
z^2&=&y\tilde g_8(x_0,y,w,v)\\
zu&=&w\tilde g_8(x_0,y,w,v)\\
u^2&=&v\tilde g_8(x_0,y,w,v)+x^2\tilde k_{10}(x_0,x,v)
\end{array}
\end{equation}
  where $v^2$ appears in $\tilde g_8$ with non-zero coefficient, which we fix to be $1$. 
  
  Conversely, if $X$ is defined by general equations of this form, then $ \omega_X = \ko_X(1)$ and $X$ is a 2-Gorenstein stable I-surface.
\end{prop}
\begin{proof}
 
The equations from Theorem \ref{thm: types and rings} which define the canonical curve of type B, can be expressed in the following format:
\[\bigwedge^2\begin{pmatrix}0&y&w&z\\x&w&v&u\end{pmatrix}=0,\  \
\begin{array}{rcl}
z^2&=&yg_8(y,v)\\zu&=&wg_8(y,v)\\u^2&=&vg_8(y,v)+x^4h_8(x,v)
\end{array}
\]
This format is useful for studying the equations of the surface $X$ which contains $C$ as canonical hyperplane section.

We reconstruct $X$ from $C$ by applying the hyperplane section principle \cite{reid90}. The ideal of relations has 9 generators and 16 syzygies, so it is possible to work out the equations of $X$ by hand apr\`es Reid \cite{reid90}. We also used the computer package of Ilten \cite{VersalDeformationsArticle} which in turn follows an algorithm laid out by Stevens \cite{stevens}. We highlight the two most important features of the result.

The $2\times4$ matrix for $C$ is weighted homogeneous with weights $\left(\begin{smallmatrix}0&2&3&5\\1&3&4&6\end{smallmatrix}\right)$ and generic entries, except that $w$ appears twice. Thus (after relabelling the entries) the only possible change to the matrix for $X$ would be to one of the $w$-entries. In fact, it is impossible to change one $w$-entry without violating some of the 16 syzygies yoking the relations together, hence the matrix for $X$ is unchanged.

The last relation output by the algorithm has the unsatisfactory form
\[u^2=v\tilde g_8(x_0,y,w,v)+x^4\tilde h_{8}+x_0x^3a_8+x_0^2x^2b_8+x_0^3xc_8\]
where $\tilde h_8(x_0,x,v)|_{x_0=0}=h_8$ and $a_8,b_8,c_8$ are general forms of degree $8$ in $(x_0,x,v)$. We write $\tilde k_{10}=x^2\tilde h_8+x_0x a_8+x_0^2b_8$ and absorb those terms of $c_8$ that are divisible by $x$ or $v$ into $x^2\tilde k_{10}$ or $v\tilde g_8$ respectively. This leaves
\[u^2=v\tilde g_8(x_0,y,w,v)+x^2\tilde k_{10}(x_0,x,v)+c x_0^{11}x,\]
where $c$ is now a constant. Using the coordinate transformations $x\mapsto x+\alpha x_0$ and $v\mapsto v+\beta x_0^4$ for appropriate choices of parameters $\alpha$, $\beta$, we eliminate $c$ to obtain the stated equation. Moreover, the new coordinates do not affect the original choice of curve section $C$. Note that $\tilde g_8(x_0,y,w,v)|_{x_0=0}=g_8(y,v)$ because after utilising the relation $w^2=yv$, all monomials of $\tilde g_8$ involving $w$ also involve $x_0$, for degree reasons.

To find $\omega_X$, we follow the same procedure as in the proof of \cite[Cor. 3.12]{CFPRR23}. That is, we use computer algebra to get the minimal free resolution of $\OO_X$ as an $\OO_\pp$-module. This is a self-dual complex of length four because the homogeneous coordinate ring of $X$ is Gorenstein, and the last term of the complex is $\OO_\pp(-23)$. Since the complex is self-dual, and $\omega_\pp=\OO_\pp(-1-1-2-3-4-5-6)=\OO_\pp(-22)$, when we apply $\mathcal{H}om(-,\omega_\pp)$, we can read off $\omega_X=\mathcal{E}xt^4(\OO_X,\omega_\pp)=\OO_X(-22+23)=\OO_X(1)$.

To show that $\omega_X^{[2]}$ is invertible, it suffices to show that $X$ does not intersect the loci $\pp(3_w,6_u)$, $\pp(4_v)$, $\pp(5_z)$ where $\OO_\pp(2)$ is not invertible. This can be deduced from the equations \eqref{eq: B2}: the monomials $v^3$ and $z^2$ appear in some equation, so $X$ is disjoint from $\pp(4)$ and from $\pp(5)$, and since $w^2$ and $u^2$ appear in separate equations, it follows that $X$ is disjoint from $\pp(3,6)$ as well. Hence $\OO_X(2)$ is invertible and $X$ is 2-Gorenstein.
\end{proof}

\begin{cor}\label{cor: B decomposition}
  Let $X$ be a $2$-Gorenstein $I$-surface with reduced canonical curve of  type $B$. Then $X = X_1\cup X_2$ is reducible and fits in a diagram
\[ 
   \begin{tikzcd}
    X_1\cup X_2 \rar{2:1} \dar[hookrightarrow] & W_1\cup W_2 \dar[hookrightarrow] \\
    \pp(1_{x_0},2_y,3_w,4_v,5_z,6_u)\cup \pp(1_{x_0}, 1_x,4_v,6_u) \dar[hookrightarrow]\rar[dashed] & 
        \pp(1_{x_0},2_y,3_w,4_v)\cup \pp(1_{x_0}, 1_x,4_v)\dar[hookrightarrow] \\
    \pp(1_{x_0},1_x,2_y,3_w,4_v,5_z,6_u)\rar[dashed]
 & \pp(1_{x_0},1_x,2_y,3_w,4_v)
   \end{tikzcd}
  \]
  where $W_1 = \{ w^2-yv =0 \} \subset \pp(1_{x_0},2_y,3_w,4_v)$ and $W_2 = \pp(1_{x_0}, 1_x,4_v) $.
\begin{enumerate}
\item The intersection 
\[
 \begin{tikzcd}
X_1\cap X_2= E\rar[hookrightarrow] & \pp(1_{x_0},4_v,6_u) \isom \pp(1_{x_0^2},2_v,3_u)
 \end{tikzcd}
\]
  is a curve of arithmetic genus one embedded via the section ring of $\ko_E(p)$, where $p = (0:1_v: 1_u)$.
  \end{enumerate}
If the defining equations are sufficiently general, then:  
\begin{enumerate}[resume]\item The pair $(X_1,E)$ is a singular Enriques surface as described in Construction \ref{constr: X1}.
  \item The pair $(X_2, E)$ is a singular K3 surface as described in Construction \ref{constr: X2}.
    \end{enumerate}
    In particular, $X$ arises as in Proposition \ref{prop: geometry B}
\end{cor}
\begin{proof}
 The defining ideal of $X$ contains the three reducible equations  $xy=xw=xz=0$ cutting out the two (weighted) linear subspaces such that $X$ is contained in their union.
 Since $X$ cannot be contained in one of them, $X = X_1 \cup X_2$ has two irreducible  components; it cannot have more by Lemma \ref{lem: X-reducible}.
 
The linear projection corresponds to the subring generated by $x_0, x, y, w, v$ and the three equations on the right hand side of \eqref{eq: B2} show that the map is of degree two on both components of $X$.

We translate the equations into the geometric descriptions.

$(i)$ The equation for $E$, \[u^2 - v\tilde g_8(x_0, 0, 0, v) = u^2 - v^3 + \dots = 0,\] is obtained by substituting $x = y  = w = z = 0$ and is exactly the Weierstrass model we get from a divisor of degree one on a curve of arithmetic genus one.

$(iii)$  Notice that $X_2=X\cap(y=w=z=0)$ is the hypersurface of weighted degree $12$ defined by $(u^2=v\tilde g_8+x^2\tilde k_{10})$ in $\IP(1,1,4,6)$. For general $\tilde g$, $\tilde k$ this has one $A_1$ singularity at $(0:0:1_v:1_u)$. The curve $E$ is cut out by $x=0$ and thus passes through the singular point.
Thus $X_2$ is a singular K3 surface as described in Construction \ref{constr: X2}.

$(ii)$ 
The other component $X_1=X\cap(x=0)$ in $\pp(1,2,3,4,5,6)$ has equations
\begin{equation}\label{eq: X1}
 \bigwedge^2\begin{pmatrix}y&w&z\\w&v&u\\z&u&\tilde g_8\end{pmatrix}=0
\end{equation}

For $\tilde g_8$ general, $X_1$ has one $A_1$ singularity at $(0:0:0:1_v:0:1_u)$.
We are going to show that $X_1$ is a singular  Enriques surface by describing explicitly the singular  K3 cover $S_1 \to X_1$, compare \eqref{Enriques diagram}. 

Consider the weighted projective space $\pp(1_{a_0},1_{a_1},2_b,4_c)$ and let $S_1$ be defined by the equation 
\begin{equation}\label{eq: invt-octic}
c^2=g_8(a_0,a_1^2,a_1b,b^2).
\end{equation}
The surface $S_1$ has $A_1$ singularities at the two points $(0:0:1:\pm1)$. By adjunction for weighted projective hypersurfaces $S_1$ is a singular K3 surface.

Let $\iota$ be the involution on $\pp(1,1,2,4)$ given by 
$(a_0:a_1:b:c)\mapsto(a_0:-a_1:-b:-c)$.
One can check that $\iota$ induces a fixed point free involution on 
$S_1$  exchanging the two $A_1$ singularities.

We compute the equations of the quotient surface $S_1/\iota$. The $\iota$-invariants are
\[x_0=a_0,\ y=a_1^2,\ w=a_1b,\ v=b^2,\ z=a_1c,\ u=bc,\ c^2\]
with respective degrees $1,2,3,4,5,6,8$.
There are several monomial relations between these generators which are analogous to those for the affine cone on the Veronese embedding of $\pp^2$ with coordinates $a_1,b,c$, hence they may be expressed as the $2\times 2$ minors of a $3\times 3$ matrix as in equation \eqref{eq: X1}, thus 
$ S_1/\iota$ is isomorphic to $X_1 $.

The projection $\IP(1,1,2,4) \to \IP(1,1)$ induces an 
elliptic pencil on $S_1$  that descends to an elliptic pencil on $X_1$.
The image of the 
$\iota$-invariant elliptic curve $S_1\cap\{a_1=0\}$  on the quotient $X_1$ is the double curve  $y=w=z=0$. 

We have thus recovered the description given in Construction \ref{constr: X2}.
\end{proof}

\subsection{Locating $\gothD_B$ inside $\overline\gothM_{1,3}$}
In this section we show that every 2-Gorenstein I-surface with reduced canonical curve of type B is $\IQ$-Gorenstein smoothable. Together with a parameter count and some further analysis we get the following.

\begin{prop} \label{prop: locating divisor B}
 Consider the closure of the set of 2-Gorenstein I-surfaces with reduced canonical curve of type B, 
\[ \gothD_B = \overline{ \left\{ [X]\in \overline\gothM_{1,3} \left| \,  \text{\begin{minipage}[c]{5cm}
 $X$ 2-Gorenstein with reduced canonical curve of type B       
\end{minipage}
 }
 \right.\right\} } \subset \overline{\gothM}_{1,3}.\]
 \begin{enumerate}
  \item The set $\gothD_B$ is an irreducible divisor in the closure of the Gieseker moduli space.
\item  The intersection of divisors  $\gothD_A \cap \gothD_B$ contains the irreducible  subset of codimension $2$
 \[\gothR_{AB}  = \{ [X] \in \gothD_B \mid \text{ $\tilde g_8$  in \eqref{eq: B2} does not contain $y^4$}\}.\]
  \end{enumerate}
 \end{prop}
\begin{proof}
Let us consider again the equations \eqref{eq: B2} defining $X$. For $\tilde g_8$ we have $13$ free parameters while for $\tilde k_{10}$ we have $21$ free parameters. So $\gothD_B$ is dominated by a linear space of dimension $34$ and thus irreducible.

From the parameter count in Proposition \ref{prop: geometry B} we see that $\dim \gothD_B = 27$.

To prove that $\gothD_B$ is in the closure of the Gieseker component, it is enough to prove that the general surface $X$ in $\gothD_B$ is $\IQ$-Gorenstein smoothable. 
 
We are going to show that the  following family $\mathcal{X} \to  \IA^1_{\lambda}$ in $\pp(1,1,2,3,4,5,6)\times \IA^1_{\lambda}$ gives a smoothing of $X$:
\[\bigwedge^2\begin{pmatrix}\lambda&y&w&z\\x&w&v&u\end{pmatrix}=0,\  \
\begin{cases}
z^2&=y\tilde g_8(x_0,y,w,v)+\lambda^2\tilde k_{10}(x_0,x,v)\\
zu&=w\tilde g_8(x_0,y,w,v)+\lambda x\tilde k_{10}(x_0,x,v)\\
u^2&=v\tilde g_8(x_0,y,w,v)+x^2\tilde k_{10}(x_0,x,v)
\end{cases}
\]
Comparing the equations of $\mathcal{X}_0$ with \eqref{eq: B2} it is clear that the central fibre describes a  surface in $\gothD_B$. 
  Now assume that $\lambda$ is invertible. Then the equations allow us to eliminate $w,v,u$:
  \begin{equation}\label{eq: general fibre}
w  = \frac{xy}{\lambda}, \; v  = \frac{xw}{\lambda}=\frac{x^2y}{\lambda^2}, \;u = \frac{xz}{\lambda} .   
  \end{equation}
  Hence the minors of the $2\times 4$ matrix vanish identically. If we define
  \[ f  = z^2 - y\tilde g_8 - \lambda^2\tilde k_{10},\]
  then after substituting using \eqref{eq: general fibre}, the remaining three equations become 
  \[ f, \; \frac{x}{\lambda}f, \; \frac{x^2}{\lambda^2}f.\]
  In other words, for $\lambda\neq 0$ the fibre
  \[ \kx_{\lambda}\isom \Proj\left( \IC[x_0,x,y,z]/f\right)\]
  is a smooth I-surface if $\lambda$ is sufficiently general, because $\tilde g_8$ and $\tilde k_{10}$ are general by hypothesis.
   The family is equidimensional over a curve and hence flat. The relative canonical bundle $\omega_{\kx/\IA^1} = \ko_{\kx/\IA^1} (1)$ is $\IQ$-Cartier, so this is  a $\IQ$-Gorenstein smoothing by \cite[Definition-Theorem 3.1]{KollarModuli}. This concludes the proof of $(i)$.
   
   To prove $(ii)$ we  assume that $\tilde g$ does not contain the monomial $y^4$ with non-zero coefficient, which is a codimension one condition on $\gothD_B$ by our parameter count. 
   Since $\tilde k_{10}$ does not contain powers of $y$, every general fibre is isomorphic to a hypersurface of degree 10 in $\IP(1,1,2, 5) $ passing through the point $(0:0:1:0)$, which are the surfaces parametrised by $\gothD_A$ by Theorem \ref{thm: surface-A}. So $\gothR_{AB} \subset \gothD_A\cap \gothD_B$ as claimed.
  \end{proof}
\begin{rem}
 The fibrewise projection of the family constructed in the proof of Proposition \ref{prop: locating divisor B} gives a family $\mathcal{W}/\IA^1_{\lambda}$ in $\pp(1,1,2,3,4)\times\IA^1_{\lambda}$ defined by
\[
\bigwedge^2\begin{pmatrix}\lambda&y&w\\x&w&v\end{pmatrix}=0
\]
of which $\kx$ is a double cover. Thus $\kw$ is a degeneration   of   $\IP(1,1,2)$ to the reducible rational surface $W = W_1 \cup W_2$ from Corollary \ref{cor: B decomposition}. 
\end{rem}

\begin{rem} 
The total space of $\mathcal{X}/\IA^1_{\lambda}$ constructed in the proof can be expressed as the $4\times4$ Pfaffians of the following extrasymmetric skew $6\times 6$ matrix, where we omit the diagonal zero-entries:
\[\begin{pmatrix}
0&\lambda&y&w&z\\
&x& w 
&v&u\\
&&z&u&\tilde g\\
&&&0&\lambda \tilde k\\
&&&&x\tilde k
\end{pmatrix}\]
On the first component of the central fibre $\kx_0\cap(x=0)$, the Pfaffians are equal to the $2\times2$ minors of the upper-right $3\times 3$ symmetric block. 
These already appeared in \eqref{eq: X1}.

Using \cite{VersalDeformationsArticle} one can check that every deformation of a general $X$ fits this format, so the versal deformation is smooth of dimension $28$, so also the moduli space should be smooth at the general point of $\gothD_B$.

Maybe this format can also be used to study the interaction with other known strata in the moduli space.
\end{rem}

\begin{rem}
 In \cite{CFPR22}  we explained how to compute the Deligne mixed Hodge structure of a stable I-surface.  If $X$ has a non-reduced canonical curve of type B, then the components of the normalisation satisfy $p_g(X_1)+p_g(X_2) = 1$ and it follows that the mixed Hodge structure on $H^2(X, \IZ)$ is not pure. In the notation of Robles \cite{Robles16}, the general such surface has Hodge type $\lozenge_{0,1}$.
 
From her results one can  also infer that no stable I-surface with Hodge type $\lozenge_{0,0}$, that is, pure Hodge structure in the second cohomology, is in $\gothD_B$. This applies in particular to the stratum $\gothR_{DD}$ described in Theorem \ref{thm: DD}.
\end{rem}
\section{Components of I-surfaces with only non-reduced  canonical curves }\label{sec:  canonical curve nonreduced}

 We now classify 2-Gorenstein I-surfaces with only non-reduced canonical curves. The method employed in Section \ref{sec: typeA} and \ref{sec: typeB}, lifting the canonical ring from the canonical curve, did not lead us to success in this case because we were unable to classifiy the relevant curves and ggs structures. Therefore, we take a more geometric approach, classifying the component surfaces directly.

 It will turn out that each component is a singular K3 surface of a particular type and we will remain with two cases, type {DD} and type {DE}. These will be treated in  Section \ref{sect: DD} and Section \ref{sect: DE} respectively.
 
\subsection{Numerical restrictions}
 Let $X$ be a 2-Gorenstein I-surface and  assume from now on  that every canonical curve $C\in |K_X|$  is non-reduced. By Proposition \ref{prop: canonical section}, $X$ is reducible 
 with normalisation $(\bar X,\bar D)=(\bar X_1, \bar D_1)\sqcup (\bar X_2, \bar D_2)$,   $\bar D_i= \bar \Gamma_i + \bar Z_i$, where $\bar \Gamma_1$  and $\bar \Gamma_2$ are identified in
  $X$ and 
 $K_{\bar X_i}+ \bar Z_i=0$. Moreover  $\bar \Gamma_1$ and $\bar \Gamma_2$ are irreducible  curves  since 
 $1= 2(K_{\bar X_i}+\bar D_i)(K_{\bar X_i}+\bar D_i)=2(K_{\bar X_i}+\bar D_i) \bar \Gamma_i$.

We write $X = X_1\cup_{\Gamma} X_2$. 
For both components consider the refinement of the standard pushout glueing diagram arising from Koll\'ar's glueing theorem \cite[Thm.~5.13]{KollarSMMP}:
\[ 
 \begin{tikzcd}
  \bar X_i \dar[swap]{\text{normalisation}} \rar[hookleftarrow] & \bar D_i = \bar \Gamma_i + \bar Z_i \dar& \bar D_i^\nu  \lar\\
  \tilde X_i \dar[swap]{\text{$S_2$-fication} }\rar[hookleftarrow] &  \tilde D_i = \tilde \Gamma_i + \tilde Z_i\dar\\
  X_i \rar[hookleftarrow] & D_i  = \Gamma + Z_i
   \end{tikzcd}.
\]
 The idea is to consider the properties of the components individually, which is problematic because they are in general not $S_2$, so not  stable log surfaces. The purpose of the $S_2$-fication   is to remedy this, so $(\tilde X_i , \tilde\Gamma_i)$ is a stable log surface with non-normal locus $\tilde Z_i$ (possibly empty). One could alternatively describe this by glueing only part of the boundary of $(\bar X_i, \bar D_i)$.

\begin{lem}\label{lem: chi}
 In the above situation we have $K_{\tilde X_i} =0$, in particular $\tilde X_i$ is Gorenstein with slc singularities. The holomorphic Euler characteristics satisfy
 \begin{equation}\label{eq: chi} \chi(X) +\chi(\tilde \Gamma_1)+ \chi(\tilde \Gamma_2) = \chi(\tilde X_1) + \chi(\tilde X_2) + \chi(\Gamma).\end{equation}
 The curves $\tilde \Gamma_1$ and $\tilde \Gamma_2$ are irreducible, and their normalisations are isomorphic. 
 \end{lem}
\begin{proof}
 Outside a subset of codimension two we have that the pullback of $K_{\tilde X_i}$ is $K_{\bar X_i} + \bar Z_i$, which is trivial by Proposition \ref{prop: canonical section}. Since $\tilde X_i$ is $S_2$, hence Cohen--Macaulay, also the canonical sheaf is $S_2$, thus trivial. 
 
 To prove the second part consider the map $ \tilde{\pi} \colon (\tilde X, \tilde \Gamma) = (\tilde X_1, \tilde \Gamma_1) \sqcup (\tilde X_2, \tilde \Gamma_2) \to X$ inducing   the diagram 
 \[
  \begin{tikzcd}
  0 \rar& \ki_\Gamma \rar\dar & \ko_X \rar\dar& \ko_\Gamma\rar\dar& 0\\
0\rar& \tilde \pi_*\ko_{\tilde X}(-\tilde \Gamma) \rar & \tilde \pi_*\ko_{\tilde  X}  \rar&\tilde \pi_*\ko_{\tilde  \Gamma}\rar& 0.
  \end{tikzcd}
 \]
The component $\Gamma$ of the non-normal locus of $X$ which connects the two components is a reduced subscheme of pure codimension one, so $\ko_\Gamma$ is $S_1$. By the depth Lemma $\ki_\Gamma$ is $S_2$ and coincides in codimension one with $\tilde \pi_*\ko_{\tilde X}(-\tilde \Gamma)$ (as in the usual case of the normalisation), so
 they are  isomorphic.   Hence
 from the additivity of the Euler characteristic   we get $$\chi(\ko_X) - \chi (\ko_{\Gamma})= \chi (\ki_{\Gamma})= \chi (\tilde  \pi_*\ko_{\tilde X}(-\tilde \Gamma)) =
 \chi( \tilde \pi_*\ko_{\tilde X}(-\tilde \Gamma)) - \chi (\tilde \pi_*\ko_{\tilde  \Gamma}),$$ which is equivalent to 
Formula \eqref{eq: chi}. 

The curve $\tilde \Gamma_i$ is irreducible because $\bar \Gamma_i$ is.
Since $\tilde \Gamma_1$ is glued to $\tilde \Gamma_2$, their normalisations are exchanged by the glueing involution and hence are isomorphic.
\end{proof}

We need the following general lemma. 
\begin{lem}\label{lem: K3-regular}
Let $ S$ be an irreducible  projective  slc surface such that $K_{ S}=0$. 
If $\chi( S)>0$, then $\chi( S)=2$. 
\end{lem}
\begin{proof} Since $p_g( S)=1$ because $K_{ S}=0$, we have $\chi( S)\le 2$ and it is enough to show that $\chi( S)=1$ cannot occur.
So assume by contradiction that $\chi( S)=1$, i.e.,  $q( S)=1$. 
Pick  a very ample line bundle $H$ on $ S$ and let $C\in |H|$ be a general curve. By the generalised Kodaira vanishing theorem (\cite[Prop. 3.1]{liu-rollenske14}),
 $h^1(\OO_{ S}(-C))=0$  and therefore the restriction sequence for $C$ induces an injection $H^1(\OO_{ S})\to H^1(\OO_C)$. It follows that the restriction map $\Pic^0( S)\to \Pic^0(C)$  has finite kernel. 
The curve  $C$, being general,  has nodal singularities, hence $\Pic^0(C)$ is a quasi-abelian variety, namely there is an exact sequence
$0\to (\IC^*)^m\to \Pic^0(C)\to A\to 0$, where $A$ is an abelian variety. So the image of $\Pic^0( S)$ in $\Pic^0(C)$ is isomorphic either to $\IC^*$ or to an elliptic curve, and therefore   $\Pic^0( S)$ contains a torsion  element $\zeta$ of order $m\ge 2$. Let $Y\to  S$ be the connected \'etale cover given by $\zeta$. Then $K_Y$ is trivial and $p_g(Y)=\chi(Y)+q(Y)-1\ge m+1-1\ge2$, a contradiction. 
\end{proof}
\begin{lem}\label{lem: types non-reduced} In the set-up of Lemma \ref{lem: chi}
 there are the following possibilites for the invariants, up to renumbering the components:
 \begin{center}
 \begin{tabular}{lccccc}
\toprule
type & $\chi(\tilde \Gamma_1)$ & $\chi(\tilde \Gamma_2)$ & $\chi(\tilde X_1)$ & $\chi(\tilde X_2)$ &$\chi(\Gamma)$ \\
\midrule
{DD} &1& 1& 2 & 2& 1 \\
{DE} & 1 & 0 & 2& 2& 0\\
\bottomrule
 \end{tabular}
 \end{center}
\end{lem}
\begin{proof}
Recall that $\Gamma$ and the $\tilde\Gamma_i$ are irreducible with $\chi=0$ or $1$. Since the maps $\tilde \Gamma_i\to \Gamma$ are birational, we also have $\chi(\Gamma)\le \chi(\tilde\Gamma_i)$, $i=1,2$. So if $\chi(\Gamma)=1$ we get $\chi(X_1)+\chi(X_2)=4$ by \eqref{eq: chi} and therefore $\chi(\tilde X_1)=\chi(\tilde X_2)=2$ by Lemma \ref{lem: K3-regular}. If $\chi(\Gamma)=0$ then \eqref{eq: chi} gives $\chi(\tilde X_1)+\chi(\tilde X_2)\ge 3$ and so again $\chi(\tilde X_1)=\chi(\tilde X_2)=2$ by Lemma  \ref{lem: K3-regular}
\end{proof}
  \subsection{Identifying components}
Next we classify the possible components. The first one is a double plane branched over a sextic. 
\begin{prop}[type D]\label{prop: type D}
 Let $( S , { \Delta})$ be an irreducible stable log surface  pair with the following properties:
 \[ K_{ S }= 0,  \  { \Delta}^2 = \frac12,  \  2{ \Delta} \text{ Cartier}, \  p_a({ \Delta}) = 0, \chi({ S })=2.\]
  Then $|2{ \Delta}|$ induces a double cover $f\colon { S } \to \IP^2$ branched over a sextic $L+B_5$, where $L$ is a line  such that $f^*L = 2{ \Delta}$ and $(\IP^2, \frac12(L+B_5))$ is lc.

For a general choice of $B_5$, the surface  $ S$ is a singular  K3 surface with five $A_1$ singularities.
  
Moreover,  we have 
  \[R( S, \Delta ) \isom \IC[x, y_1, y_2, z]/\left(z^2 - f_{10}(x, y_1, y_2)\right)\]
  with $\deg (x, y_1, y_2, z) = (1,2,2,5)$.
\end{prop}
\begin{proof}
Riemann--Roch and generalised Kodaira vanishing give $h^0(2{ \Delta})=\chi(2{ \Delta})=3$. 
 The restriction sequence  together with Kodaira vanishing gives an exact sequence
 \[  0\to H^0({S}, { \Delta}) \to H^0({S}, 2{ \Delta}) \to H^0({ \Delta}, \ko_{ \Delta}(2{ \Delta})) \to 0\]
 Since ${ \Delta}\isom \IP^1$ we see that $|2{ \Delta}|$ has no base points, because it has no base points on ${ \Delta}$.

 Since ${ \Delta}$ is ample and $(2{ \Delta})^2 = 2$,  the system $|2{ \Delta} |$ induces  a finite map of degree two $ {S} \to \IP^2$ that maps 
  ${ \Delta}$ isomorphically to a line $L$, which is necessarily a component of the branch locus. The total degree of the branch locus and the fact that $(\IP^2, \frac12(L+B_5))$ is  lc follow from the Hurwitz formula (cf. \cite[Prop.~2.5]{alexeev-pardini12}).

  Concerning the ring $R(S, \Delta)$, we can write   $H^0(S, \Delta)= \langle  x \rangle $  and  $H^0(S, 2 \Delta)= \langle  x^2, y_1,y_2 \rangle $. 
  We see that  $h^0(S,m\Delta) = 2 +\frac14 m^2$ if $m$  is even  and  restricting to $\Delta$ (which has genus  0),  we get $h^0(S,m\Delta) =1+\frac14 (m^2  -1 )  $ if $m$ is odd. 
  
  So we need  a further generator $z$ in degree 5, and we see that the first relations are in degree 10 and we can assume  that it is of the form  $z^2 =f _{10}(x, y_1, y_2)$, then 
  considering the map $\IC[x, y_1, y_2, z]/\left(z^2 - f_{10}(x, y_1, y_2)\right) \to R(S,\Delta)$ and looking at the Hilbert series we can conclude that they are isomorphic. 
  
 The double cover 
 $S\to\pp^2\cong\pp(1,2,2)$ corresponds to the truncation in degree 2 and  its branch locus consists of  the  curve $B_5$ defined by $\{f_{10}(x, y_1,y_2)=0\}$ together with the 
  distinguished line $L$ defined by $\{x=0\}$.
  \end{proof}

\begin{prop}[type E]\label{prop: type E}
 Let $(S , \Delta)$ be an irreducible stable log surface with $\Delta$ irreducible and the following properties:
 \[ K_S =0, \ \Delta^2 = \frac12,  \ 2\Delta \text{ Cartier},  \ p_a(\Delta) = 1, \ \chi(S ) = 2.  \]
 Then $|4\Delta|$ induces a double cover $f\colon S  \to \kc_4 \subset \IP^5$, where $\kc_4$ is the cone over the rational normal curve of degree four. The morphism $f$ is branched over the vertex and a cubic section $B$ such that $(\kc_4, \frac 12 B)$ is lc.
 
  For general choice of $B$, we get a singular elliptic  K3 surface with one $A_1$ singularity over the vertex of the cone, which is also a base point for the elliptic pencil $|\Delta|$ (compare Construction \ref{constr: X2}). 
 
 Moreover,  we have 
  \[R(S, \Delta) \isom \IC[u_0, u_1, v, w]/\left(w^2 - g_{12}(u_0,u_1,v)\right)\]
  with $\deg (u_0, u_1, v, w) = (1,1,4,6)$.
\end{prop}
\begin{proof}
Riemann--Roch and generalised Kodaira vanishing give $h^0(2\Delta)=\chi(2\Delta)=3$.
The restriction sequence  together with Kodaira vanishing gives an exact sequence
 \[  0\to H^0(S , \Delta) \to H^0(S , 2\Delta) \to H^0(\Delta, \ko_\Delta(2\Delta)) \to 0.\]
Since $h^0(\Delta, \ko_\Delta(2\Delta))=1$, we get $h^0(\Delta)=2$. All the curves of  $|\Delta|$ are irreducible and contain the point $p$ defined by  $\OO_{\Delta}(2\Delta)=\OO_{\Delta}(p)$. 

 Similarly, one has  $h^0(4\Delta)=6$ and an exact sequence
\[  0\to H^0(S , 3\Delta) \to H^0(S , 4\Delta) \to H^0(\Delta, \ko_\Delta(4\Delta)) \to 0\]
 Since $p_a(\Delta) = 1$ and $4 \Delta^2 = 2$, we see that $|4\Delta|$ has no base points, because it has no base points on $\Delta$. Its restriction to the curves of $|\Delta|$ defines a double cover $\Delta \to \IP^1$, so the image $\Sigma$ of the morphism $\varphi$ given by $|4\Delta|$ is ruled by lines and $\deg \varphi\ge 2$. More precisely,  $\Sigma$ is a cone, since the images of the curves of $\Delta$ all go through the point $\varphi(p)$. Since $8=(4 \Delta)^2 = \deg \varphi\deg \Sigma\ge 2\cdot 4$, we have $\deg\varphi=2$ and $\Sigma$ is the  cone $\kc_4$ over the rational normal curve of degree 4. For every $\Delta'\in |\Delta|$ we see that $|4\Delta|\restr{\Delta'}$ contains $2p$, so $\varphi$ is branched over $p$. Finally the Hurwitz formula shows that $\varphi$ is branched on a divisor
 of $|-2K_{\kc_4}|=|\OO_{\kc_4}(3)|$.
  
  Concerning the ring $R(S, \Delta)$, we can write   $H^0(S, \Delta)= \langle  u_0,u_1 \rangle $  and  
  restricting to $\Delta$ we see that we need  further generators $v, w$ in degree 4, resp. 6.  
 The double cover 
 $\ S\to\kc_4\cong\pp(1,1,4)$ corresponds to the truncation in degree 4 and  its branch locus consists of the vertex of the cone and  the  curve $B=\{g_{12}(u_0,u_1,v)=0\}$.
  This  corresponds to the  only relation   we need.   
\end{proof}

 \section{I-surfaces of type {DD}}\label{sect: DD}
 The surfaces of type {DD} were already described in \cite[Example 4.7]{FPR17a}, so we can be brief here.
 Let  $\gothR^{{DD}}$ be the closure of the locus of such surfaces in the moduli space. 
 \begin{thm}\label{thm: DD}
  Let $X$ be a 2-Gorenstein I-surface such that  every canonical curve  is non-reduced and the $S_2$-fication of both components is of  type D (compare Lemma \ref{lem: types non-reduced}).
  
  Then $X$ is a complete intersection of degree  $(2,10)$ in $\IP(1,1,2,2,5)$ and is  smoothable.
   More precisely 
  \[ R(X, K_X) \isom \IC[x_1, x_2, y_1, y_2, z]/(x_1x_2, z^2 - f_{10}(x_1, x_2, y_1, y_2)).\]
 The set $\gothR_{{DD}}$ is irreducible of dimension 26 and satisfies
 \[ \gothR_{{DD}} \subset \gothD_A \subset \overline\gothM_{1,3}^\text{class}.\]
 \end{thm}
\begin{proof}
 It is easy to check, that the linear system $|2K_X|$ defines a double cover of the union of two planes in $\IP^3$ branched, as in the smooth case,  over a quintic section.
 
 These surfaces give a $26$-dimensional locus inside the moduli space: 
 the linear system of quintics in the two planes that match on the intersection line is of dimension $35$ and the automorphism group of  the union of two planes in $\IP^3$  has dimension $9$. 
 (cf. \cite[Example 4.7]{FPR17a}). 
 
 To show that $\gothR_{{DD}}\subset \gothD_A$, we write in the equation
 \[ f_{10} \equiv  g(y_1, y_2) \mod (x_1, x_2).\]

 We may assume after a coordinate change that  $y_1$ is a factor of $g$. Considering the family 
 $\kx/\IA^1_\lambda$ given by equations
 \[ x_1x_2 - \lambda y_1= z^2 - f_{10}=0\]
 we see that the general fibre is in $\gothD_A$ by Theorem \ref{thm: surface-A}. 
 \end{proof}
 
 \section{I-surfaces of type {DE}}\label{sect: DE}
We now treat the case of 2-Gorenstein I-surfaces with only non-reduced canonical curves  of type {DE} defined in Lemma \ref{lem: types non-reduced}. Let us consider the set of these surfaces
 \[ \gothU^{{DE}} = \{[X] \in \overline\gothM_{1,3} \mid \text{$X$ of type {DE}}\}.\]
in the moduli space of stable surfaces.

The results of this section are summed up in the following result, where we relegate the precise geometric and algebraic descriptions to the subsequent subsections.
\begin{thm}\label{thm: DE}
Every surface $X = X_1 \cup X_2$ of type {DE} is glued from two particular singular K3 surfaces along a nodal  rational curve. It is canonically embedded in $\pp(1,1,2,2,3,4,5,6,7)$ and its ideal is as in \eqref{eq: ideal DE}.
 
 The subset $\gothU^{{DE}}$ is irreducible of dimension $30$ and does not intersect the closure of the Gieseker component. 
Its closure 
\[ \overline\gothM^{{DE}}  = \overline{\gothU^{{DE}}}\subset \overline \gothM_{1,3}\]
 is an irreducible component of the moduli space. 
 \end{thm}

 The proof of Theorem \ref{thm: DE} will occupy the rest of the section and follow from Propositions \ref{prop: DE geometry}, \ref{prop: DE moduli} and \ref{prop: algebraic DE}.

 \subsection{Geometric description and moduli}
 Let $X$ be a 2-Gorenstein I-surface with only non-reduced canonical curves of type DE. Then $X = X_1\cup_\Gamma X_2$ and  Propositions \ref{prop: type D} and \ref{prop: type E} describe the $S_2$-fication of the components as stable pairs $(\tilde X_i, \tilde \Gamma_i)$. Incorporating the information $\chi(\tilde \Gamma_1)  = 1 >\chi(\Gamma) =0  = \chi(\tilde \Gamma_2)$ we can sum up the information in the following diagram:
\begin{equation}\label{eq: glueing diagram DE}
 \begin{tikzcd}[row sep = small]
  \IP^2 && \arrow{ll}{\text{over } L+B_5} [swap] {2:1, \text{ br.}}\tilde X_1 \arrow{dr} &  & \tilde X_2 \arrow{dl}\arrow{rr}{2:1, \text{ br. over}}[swap]{B_3+\text{vertex}}  && \IP(1,1,4)\\ 
  &&&  X_1 \cup_\Gamma X_2\\ 
  L \arrow[hookrightarrow]{uu} && \arrow{ll}[swap]{\isom} \tilde \Gamma_1 \arrow[hookrightarrow]{uu} \arrow{dr}[pos = .3]{\text{bir.}} && \tilde \Gamma_2\arrow{dl}[swap, pos = .3]{\isom} \arrow[hookrightarrow]{uu} \arrow{rr}{2:1, \text{ br. over}}[swap]{\text{4 points}} && \IP(1,4)\arrow[hookrightarrow]{uu}\\
  &&& \Gamma\arrow[hookrightarrow]{uu}\\
 \end{tikzcd}
\end{equation}
 \begin{prop}\label{prop: DE geometry}
 Let $(\tilde X_1, \tilde \Gamma_1)$ be of type D and $(\tilde X_2, \tilde \Gamma_2)$ of type E. 
 \begin{enumerate}
  \item The pairs  $(\tilde X_1, \tilde \Gamma_1)$  and $(\tilde X_2, \tilde \Gamma_2)$ occur as the $S_2$-fications of the components of a 2-Gorenstein I-surface $X$ if and only if
  \begin{enumerate}
   \item $\tilde \Gamma_2$ is a nodal rational curve (of arithmetic genus 1),
   \item the line $L$ is bitangent to $B_5$, that is, $B_5|_L = r + 2(s_1+s_2)$, where $r,s_1$ and $s_2$ are distinct points.
  \end{enumerate}
\item The general $\tilde X_1$ as in $(i)$ is a singular K3 surface with exactly one  $A_1$ singularity at $r$ and $A_3$ singularities at $s_1$ and $s_2$. 

The general $\tilde X_2$ as in $(i)$ is  a singular K3 surface with a unique $A_1$ singularity at a point $p$ mapping to the vertex in $\kc_4 = \IP(1,1,4)$. 
  \item   The map $\tilde \Gamma_1 \to \Gamma \isom \tilde \Gamma_2$ and therefore also the $S_2$-fication map $\tilde X_1 \to X_1$, identifies the two points $s_1$ and $s_2$ as shown in Figure \ref{fig: component DE}.
 \end{enumerate}
 \end{prop}

 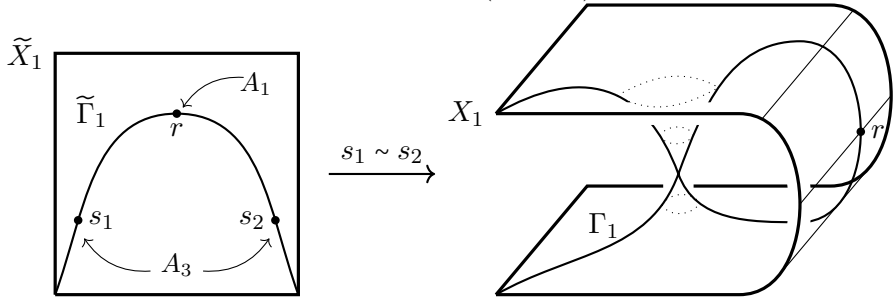
\begin{figure}
  \caption{$S_2$-fication of the (generic) component $X_1$}\label{fig: component DE}
  \begin{tikzpicture}[scale = .8]
   
   \begin{scope}[xshift = -7.25cm] %
    \draw[very thick]
    (0,0) -- (4,0) -- ( 4,4) -- (0,4) node [left] {$\tilde X_1$}
    to (0,0);

    \draw[thick] (0,0) to [in = 180, out = 70]
    coordinate [pos = .3] (s1)
    node [ pos = .7, above left] {$\tilde \Gamma_1$}
    (2,3) 
    coordinate  (r)
    to[in = 110, out = 0]
    coordinate [pos = .7] (s2)
    (4,0);
    
    \fill (s1) circle [radius = 2pt] node[right] {$s_1$};
    \fill (s2) circle [radius = 2pt] node[left] {$s_2$} ;
    \fill (r) circle [radius = 2pt]  node[below] {$r$};
    
    \draw[->, thick] (4.5, 2) to node [above] {$s_1\sim s_2$} ++ (1.75,0);
    
    \node (A2) at (2,.5) {\small $A_3$};
    \draw[<-]  (s2) ++( -.1, -.3) to[bend left] (A2);
    \draw[<-]  (s1) ++( .1, -.3) to[bend right] (A2);
    
    \draw[<-] (r) ++ (.1, .1) to[bend left] ++ (1, .5);
  \draw (r) ++ (1.3, .5) node  {\small $A_1$};
   \end{scope}

   \begin{scope}[scale = .5]

\draw[very thick] (0,0) to ++ (3,3.6) to  ++ (8,0)
to[out = 0, in = -90] ++ (2,3)
 to[out = 90, in  = 0] ++ (-2,3) -- ++ (-8,0) to ++ (-3,-3.6);

\draw [line width = .3cm, white] 
(5.5, 3.6) to ++ (1,0);

\coordinate (P) at (6, 4);
\draw[thick] (0,0) 
to[out = 30, in = -110] 
coordinate [ pos = .85] (P1)
node[pos = .5, above] {$\Gamma_1$}
(P) to[ out = 70, in = -180] 
 coordinate [ pos = .2] (P2)
  coordinate [ pos = .47] (P5)
(10, 8.4) 
 to[out = 0, in = 90]
 ++ (2,-3)  coordinate (rr)
 to[out = -90, in  = 0] 
 ++ (-2,-3) 
 to [out = 180, in = -70]
 coordinate [ pos = .73] (P3)
 (P)
 to [out  = 110, in = 30, looseness = 1.3] 
 coordinate [ pos = .16] (P4)
  coordinate [ pos = .44] (P6)
 (0,6)
;

\fill (rr) circle [radius = 4pt] node[right] {$r$};

\begin{scope}[very thick]


\draw [line width = .3cm, white] 
(2,6) to (8,6)
to[out = 0, in = 90]
++ (2,-3) 
to[out = -90, in  = 0] ++ (-2,-3);

\draw (0,0) -- ++ (8,0)
to[out = 0, in = -90]
coordinate [pos = 0.4] (Q1)
++ (2,3)  
to[out = 90, in  = 0] coordinate [pos = 0.8] (Q2)
++ (-2,3)
-- ++ (-8,0)
node[left] {$X_1$};

\end{scope}

\begin{scope}[thin]
\draw (8,0) ++ (2,3)  to ++ (3, 3.6);
\draw (Q1) to ++ (3, 3.6);
\draw (Q2) to ++ (3, 3.6);

%
%
\begin{scope}[dotted]
\draw (P1) to[out = -60, in = -120] (P3);
\draw (P1) to[out = 60, in = 120] (P3);
\draw (P4) to[out = -60, in = -120] (P2);
\draw (P4) to[out = 60, in = 120] (P2);

\draw (P6) to[out = -30, in = -150] (P5);
\draw (P6) to[out = 30, in = 150] (P5);
\end{scope}

\end{scope}

\end{scope}
\end{tikzpicture}
\end{figure}

 \begin{proof} We use Koll\'ar's glueing principle \cite[Thm.~5.13]{KollarSMMP}.
 In order to be able to glue the two surfaces, the normalisations of the boundaries have to be isomorphic, so $\tilde \Gamma_2^\nu \isom \tilde \Gamma_1^\nu = \tilde \Gamma_1 \isom \IP^1$. Since the pair is slc we have that $\tilde \Gamma_2$ is a nodal curve of arithmetic genus 1.
 
 The glueing involution in Koll\'ar's glueing principle also has to preserve the different. On $\tilde \Gamma_2^\nu$, we have $\Diff_{\tilde \Gamma_2^\nu}(0) =\frac 12 p + q_1 + q_2 $, where $p$ is the basepoint of $|\tilde \Gamma_2|$, an $A_1$ singularity in $\tilde X_2$, and the $q_i$ are the preimages of the node of $\tilde \Gamma_2$. 
 
 Let us compute the different on $\tilde \Gamma_1$. Let $\iota \colon\tilde \Gamma_1 \to \tilde X_1$ be the inclusion and recall that $\tilde \Gamma_1$ is a smooth rational curve and that 
 \[\omega_{\tilde X_1}(\tilde \Gamma_1)\refl{2} = \omega_{\tilde X_1}\refl{2}(2\tilde \Gamma_1) =  f^*\ko_{\IP^2}(2K_{\IP^2}+ B_5+L)(f^*L)=f^*\ko_{\IP^2}(2K_{\IP^2}+ B_5+2L)\]
 is locally free.
 Let $U \subset \tilde X_1$ be the smooth locus and note that the generic point of $\tilde \Gamma_1$ lies in $U$. 
 
 Then by definition \cite[(4.3.7), p.154]{KollarSMMP} the divisor  $2\Diff_{\tilde \Gamma_1}(0)$ is uniquely defined by the property that the square of the Poincar\'e residue morphism $\omega_{U}(\tilde \Gamma_1\cap U) |_{\tilde\Gamma_1\cap U }\isom \omega_{\tilde\Gamma_1\cap U}$ extends to an isomorphism
 \[ \omega_{\tilde \Gamma_1}^{\tensor 2}\left( 2\Diff_{\tilde \Gamma_1}(0)\right)
 \isom \iota^*\!\omega_{\tilde X_1}(\tilde \Gamma_1)\refl{2} \isom \iota^*\!f^*\ko_{\IP^2}(2K_{\IP^2}+ B_5+2L)
 \]
 In total, we get $\Diff_{\tilde \Gamma_1}(0) = \frac 12B_5|_L$ under the isomorphism $\tilde \Gamma_1\isom L$, so compatibility for glueing forces $(ii)$.

 Conversely, $\Aut(\IP^1)$ acts 3-transitively, so any pair with the above properties can be glued, thus we have proved $(i)$.
  The last two points follow easily.
  \end{proof}
\begin{rem} \label{rem: switch}
In order to glue surfaces of type D and E as in Proposition \ref{prop: DE geometry}, we need to pick an isomorphism $\tilde \Gamma_1\to \tilde \Gamma_2^{\nu}$ that preserves the different. This boils down to mapping $s$ to $p$ and the points $s_1,s_2$ to $q_1,q_2$, so we have two choices. However, acting  on $\tilde X_1$  with the covering involution and on $\tilde X_2$ with the identity gives an isomorphism between the surfaces corresponding to the two choices.
\end{rem}

\begin{lem} \label{lem: DE no involution}
Let $X=X_1\cup X_2$ be a 2-Gorenstein I-surface of type DE. 
If   $\sigma$ is an involution of $X$  that  acts as the identity on $|mK_X|$ for $m\le 4$, then $\chi(X/\sigma)\ge 2$.
\end{lem}
\begin{proof}
Let $\sigma$ be an involution as in the statement and consider the  induced  involution of $\tilde X=\tilde X_1\sqcup \tilde X_2$: since $\tilde \Gamma_1$ and $\tilde \Gamma_2$  are not isomorphic by Proposition \ref{prop: DE geometry},  they  cannot be exchanged by the involution and so   $\sigma$ induces involutions $\sigma_i$ of $\tilde X_i$   that preserve $\tilde \Gamma_i$, $i=1,2$. 

 Let  $\alpha \in H^0(\tilde X_1, \tilde \Gamma_1)$ be a  section vanishing on $\tilde \Gamma_1$.   
 There is an injective  pull back map $r =(r_1,r_2)\colon H^0(4K_X)\to H^0(\tilde X_1, 4\tilde \Gamma_1)\oplus H^0(\tilde X_2, 4\tilde \Gamma_2)$, so we may identify $H^0(4K_X)$ with the image of $r$.
So $H^0(4K_X)$ contains all the sections of the form $(\alpha^2\gamma,0)$, where $\gamma$ varies in $H^0(2\tilde \Gamma_1)=\pi_1^*H^0(\pp^2,\OO_{\pp^2}(1))$, where $\pi_1$ is the double cover given by $|2\tilde \Gamma_1|$ (cf. Proposition  \ref{prop: type D}). So $\sigma_1$ is either the identity or the involution associated with $\pi_1$ and, in either case it restricts to the identity on $\tilde \Gamma_1$. As a consequence  $\sigma$ restricts to the identity on $\Gamma$, and therefore
also $\sigma_2$ restricts to the identity on $\tilde \Gamma_2$.
An argument similar to the previous one shows that $\sigma_2$ preserves the curves of $|\tilde \Gamma_2|$ and acts on $\kc_4\subset \pp^5$, the image of the 2-to-1 map $\pi_2$ defined by $|4\tilde\Gamma_2|$ (cf. \ref{prop: type E}),  mapping each ruling to itself. Let $\tau$ be the   automorphism  on $\kc_4\cong \pp(1,1,4)$ induced by $\sigma$.  If $\tau$ is the identity then $\sigma_2$ is  either the identity or the covering involution. Since $\tilde \Gamma_2$ is not in the branch locus of $\tilde X_2\to \kc_4$, the covering involution is not the identity on $\tilde \Gamma_2$, so  in this case $\sigma_2$ must be the identity. Now assume that $\tau$ is a non-trivial involution of  $\kc_4$: 
since $\tau$ maps each ruling to itself, there are weighted  homogeneous coordinates $x_0,x_1, y$  on $\pp(1,1,4)$ such that $\tau$ is given by $(x_0,x_1,y)\mapsto (x_0,x_1,-y)$. So $\tau$ is non-trivial on every  ruling of $\kc_4$ and, a fortiori, $\sigma_2$ is not trivial on $\tilde \Gamma_2$. So the only possibility is that $\sigma_2$ is the identity, and therefore if $\sigma$ is non-trivial then  $\sigma_1$ is the covering involution of $\tilde X_1\to\pp^2$.

If $\sigma$ is the identity,  then $\chi(X/\sigma)=\chi(X)=3$. If $\sigma$ is not the identity, then  it is easy to see that $Y:=X/\sigma$ is defined by the push-out diagram:
\begin{equation}\nonumber
\begin{tikzcd}
    \tilde X/\sigma \dar\rar[hookleftarrow] & \tilde \Gamma/\sigma \dar & \tilde  \Gamma^\nu/\sigma \lar[swap]{\bar\nu}\dar
    \\
Y \rar[hookleftarrow] &\Gamma &\Gamma^\nu.\lar[swap]{\nu}
    \end{tikzcd}
\end{equation}
Now Lemma \ref{lem: chi} gives
$1=\chi(Y)=\chi(\Gamma)+\chi(\tilde X/\sigma)-\chi(\tilde \Gamma)=0+(1+2)-(1+0)=2$.
\end{proof}

\begin{prop}\label{prop: DE moduli}
  The subset $\gothU^{DE}$ is irreducible of dimension $30$ and does not intersect the closure of the Gieseker component. 
Its closure 
\[ \overline\gothM^{DE}  = \overline{\gothU^{DE}}\subset \overline \gothM_{1,3}\]
 is an irreducible component of the moduli space. 
\end{prop}
\begin{proof}
First  we count the moduli of the construction. Since the elliptic fibration of a surface of type E has finitely many singular fibres, that are all nodal if the surface is general,  by Proposition \ref{prop: DE geometry} it is enough to count parameters for surfaces of type E and   for  surfaces of type 
 H such that the  double cover ${\tilde X_1 } \to \IP^2$ is branched over a sextic $L+B_5$ with the line $L$  bitangent to the quintic $B_5$. 
 Fixing the  line $L = \{x_0=0\}$  and  points $p,q_1,q_2\in L$,  we see that   system of plane  quintics tangent to $L$ at $q_1$ and $q_2$ and passing through $p$  is has dimension 15. Subtracting the dimension of the subgroup of $\Aut(\pp^2)$ that fixes $L$ pointwise we obtain $15-3=12$ moduli.
 
 Surfaces of type E are double covers of $ \kc_4 \cong \IP(1,1,4)$, branched over a  cubic section  $B$ and the vertex.  The curve $B$ moves in a system of dimension 27, so subtracting the dimension of $\Aut(\kc_4)$ we get $27-9=18$ moduli. Summing up, surfaces of type DE depend on $12+18=30$ moduli.
 We can consider the surfaces of type E such that the branch locus of the map to $\kc_4$ is simply tangent to a fixed ruling  $R_0$ and define $\tilde \Gamma_2$ to be the preimage of $R_0$; since $\Aut(\kc_4)$ acts transitively on the set of rulings,  in this way we obtain an irreducible family of pairs $(\tilde X_2, \tilde\Gamma_2)$ that contains every isomorphism class of such pairs. We can take the double cover of this family obtained by labelling the preimages of  the node of $\tilde\Gamma_2$ in the normalization map $\tilde \Gamma_2^{\nu}\to \tilde \Gamma_2$. By Remark \ref{rem: switch}, an irreducible  component of this double cover contains all the isomorphism classes of pairs of type E + labelling of the preimages of the node, and therefore surfaces of type DE give an irreducible locus  $\gothU^{DE} \subset\overline{\gothM}_{1,3}$ of dimension 30. 
 
 Now let $X$ be a surface of type DE and assume that there is a one parameter deformation $\kx\to C$ over a smooth curve such that the general fiber is not of type DE. Since the Cartier index is lower semicontinuous,  because $\omega_{\kx/C}\refl{m}$ being locally free is an open condition, up to shrinking $C$ we may assume that all fibres of $\kx$ are 2-Gorenstein and not of type DE except the central fibres. Then by our classification results (Theorem \ref{thm: four types}) the fibres of $X$ are smoothable, so by picking another family we may assume that every fibre except $X$ is smooth.

 The general fibre $X_t$ thus carries an involution
$\sigma_t$, induced by the bi-canonical map, which acts trivially on  $m$-canonical system for $m\le 4$ and such that $\chi(X_t/\sigma_t) = 1$.  By \cite[Prop. 2.6]{FPRR22}, the involutions on the general fibres extend to give a global fibrewise involutionon  on the family $\kx/C$, so there is an involution $\sigma$ on $X$ contradicting Lemma \ref{lem: DE no involution}. 

Therefore any small deformation of a surface of type DE is again a surface of type DE, the locus $\gothU^{DE}$ does not intersect the closure of the Gieseker component and  its closure $\overline\gothM^{DE}$ is an irreducible component of the moduli space. 
\end{proof}

 \subsection{Algebraic description}
  Here we reverse the approach we used in the non-reduced case, and deduce the equations of the surfaces of type DE from their geometric description via Koll\'ar's glueing. 
More precisely, by the glueing principle for pluricanonical sections \cite[Prop.~5.8]{KollarSMMP} the glueing diagram \eqref{eq: glueing diagram DE} lets us compute the canonical ring as a pullback ring in the diagram
\[\begin{tikzcd}
R(X, K_X) \dar[hookrightarrow]{\tilde \pi^*} \rar & R(\Gamma^\nu, K_X|_{\Gamma^\nu})\dar[hookrightarrow]\\ 
R(\tilde X_1, \tilde \Gamma_1) \times R(\tilde X_2, \tilde \Gamma_2) \rar & R(\tilde \Gamma_1^\nu, \tilde \Gamma_1|_{\tilde \Gamma_1^\nu})
\times
R(\tilde \Gamma_2^\nu, \tilde \Gamma_2|_{\tilde \Gamma_2^\nu}),
\end{tikzcd}
\]
where we use the superscript $(-)^\nu$ to denote normalisation and have incorporated the fact that $K_{\tilde X_i} = 0$. 

Since  the glueing involution $\tau$ identifies $\tau(\tilde \Gamma_1 )\isom \tilde \Gamma_2^\nu \isom \Gamma^\nu$ we can rewrite the pullback diagram as 
\begin{equation}\label{eq: pullback ring}
\begin{tikzcd}
R(X, K_X) \dar \rar & R(\tilde X_1, \tilde \Gamma_1)\dar{\alpha^*}\\ 
 R(\tilde X_2, \tilde \Gamma_2) \rar{\beta^*}& R(\Gamma^\nu, K_X|_{\Gamma^\nu}).
\end{tikzcd}
\end{equation}

To make the above explicit, we fix algebraic descriptions of the $S_2$-fication of the two components. 
\begin{lem}\label{lem:  X_1 pinched}
Let $(\tilde  X_1,\tilde \Gamma_1)$ be a surface of type D with $B_5$ bitangent to $L$ as in Proposition \ref{prop: DE geometry}. Then  we can choose generators for the section ring $R(\tilde X_1, \tilde \Gamma_1)$ such that 
\begin{equation}\label{eq: X_1 tilde}
\begin{split}
\tilde  X_1\colon& \quad \left\{z^2=y_0(y_1+y_0)^2(y_1-y_0)^2+x^2f_8(x,y_0,y_1)\right\}\subset \IP(1_x,2_{y_0},2_{y_1}, 5_z) \\
\tilde \Gamma_1 \colon & \quad   \tilde  X_1\cap\{x=0\}\\
& \quad r=(0:0:1:0),\, s_1=(0:1:1:0),\, s_2=(0:1:-1:0)
\end{split}
\end{equation}
where $f_8(x,y_0,y_1)$ is homogeneous of weighted degree $8$.

Conversely, every sufficiently general $f_8$ will define a pair $(\tilde X_1, \tilde \Gamma_1)$ as in Proposition \ref{prop: DE geometry}.
\end{lem}
\begin{proof} As shown in Proposition  \ref{prop: type D}, $\tilde  X_1$ is a hypersurface $z^2=f_{10}(x,y_0,y_1)$ in $\pp(1,2,2,5)$. 
The branch curve of the double cover $\tilde  X_1\to\pp^2\cong\pp(1,2,2)$ consists of the curve $B_5$ defined by $\{f_{10}=0\}$ together with the distinguished line $L$ defined by $\{x=0\}$. 
The bitangency condition means that the quintic $f_{10}|_{x=0}$ has two double roots, and after a coordinates change we may write $f_{10}|_{x=0}=y_0(y_1+y_0)^2(y_1-y_0)^2$.
\end{proof}

\begin{lem}\label{lem: X_2 pinched}
Let $(\tilde  X_2,\tilde \Gamma_2)$ be a surface of elliptic type E with $\tilde \Gamma_2$ a nodal rational curve as in Proposition \ref{prop: DE geometry}. Then  we can choose generators for the section ring $R(\tilde X_2, \tilde \Gamma_2)$ such that 
\begin{equation}\label{eq: X_2 tilde}
\begin{split}
\tilde  X_2\colon& \quad \left\{w^2=v(v-u_1^4)^2+u_0g_{11}(u_0,u_1,v)\right\} \subset \IP(1_{u_0},1_{y_1},4_{v}, 6_w) \\
\tilde \Gamma_2 \colon & \quad   \tilde  X_1\cap\{u_0=0\} = \left\{w^2=v(v-u_1^4)^2\right\}\subset\IP(1_{y_1},4_{v}, 6_w)  \\
& \quad q=(0:1:1:0),\, p=(0:0:1:1),
\end{split}
\end{equation}
where $g_{11}(u_0,u_1,v)$ is homogeneous of weighted degree $11$, $p$ is the base point of the ellipitic pencil $|\tilde \Gamma_2|$, and $\tilde \Gamma_2$ has a node at $q$. 

Conversely, every sufficiently general $g_{11}$ will define a pair $(\tilde X_2, \tilde \Gamma_2)$ as in Proposition \ref{prop: DE geometry}.

\end{lem}

\begin{proof} As shown in Proposition  \ref{prop: type E}, $\tilde X_2$ is a hypersurface $w^2=g_{12}(u_0,u_1,v)$ in $\pp(1,1,4,6)$. 
The elliptic pencil  is spanned by $\langle u_0,u_1\rangle$ and we suppose that the nodal rational curve $\tilde \Gamma_2$ is defined by $\{u_0=0\}$.
 Since $\tilde \Gamma_2$ is nodal, it follows that $g_{12}|_{u_0=0}$ has a double root and we choose coordinates so that $g_{12}|_{u_0=0}=v(v-u_1^4)^2$. The  node  $q$ of $\tilde \Gamma_2$ corresponds to this double root.
\end{proof}
Having identified the components, we now turn to the glueing.
\begin{lem}
Identify $\Gamma^\nu \isom \IP^1\isom  \IP(1_t,2_s)$. Then we can choose the glueing invlution $\tau$ such that the maps $\alpha^*$ and $\beta^*$ in \eqref{eq: pullback ring} are, with respect to the coordinates chosen above, induced by 
\[ \begin{tikzcd}[row sep = tiny] 
\alpha\colon  \pp(1,2)   \rar & \pp(1,2,2,5)   \\  
(t,s)  \rar[mapsto] & (0,t^2,s,t(s^2-t^4))    \end{tikzcd}  \qquad \begin{tikzcd} [row sep = tiny] 
\beta\colon  \pp(1,2)   \rar & \pp(1,1,4,6)   \\  
(t,s)  \rar[mapsto] & (0,t,s^2,s(s^2-t^4) ) \end{tikzcd}.
\]
\end{lem}
\begin{proof}
 By definition, the image of $\alpha$ is $\tilde \Gamma_1\subset \tilde X_1$ and the image of $\beta$ is $\tilde \Gamma_2\subset \tilde X_2$. 
 
 By Remark \ref{rem: switch} the choice of involution does not matter as long as it identifies the preimages of the node $q\in \tilde \Gamma_2$ with $s_1$ and $s_2$ and the point $p$ with the point $r$. 
 This is satisfied here, because 
\begin{align*}
&\alpha(1:\pm1)=(0:1:\pm1:0)=s_i &\text{ and }&&\beta(1,\pm1)=(0:1:1:0)=q,\\ &\alpha(0:1)=(0:0:1:0)=r &\text{ and } &&\beta(0:1)=(0:0:1:1)=p.
\end{align*}
Note that our choice of grading respects the gradings of the rings on the components, which concludes the proof. 
\end{proof}

These descriptions and \eqref{eq: pullback ring} now allow us to compute the canonical ring 
\[R(X, K_X) = \{(\xi,\eta)\in R(\tilde X_1, \tilde \Gamma_1)\times R(\tilde X_2, \tilde \Gamma_2) \mid \alpha^*(\xi)=\beta^*(\eta)\}.\]
of $X$ explicitly, but it turns out that it is convenient to treat the components $X_i \subset X$ first individually. 

It is elementary but tedious   to check that 
a minimal set of generators for $R(X, K_X)$ is listed in Table \ref{tab: gens DE}. Consequently, we have  $X_1=X\cap\{a_1=b_1=0\}$ and $X_2=X\cap\{a_0=c=0\}$.
\begin{table}[ht]
 \caption{Generators of the canonical ring of type DE and their restriction to the components}
 \label{tab: gens DE}
 \begin{tabular}{ccccc}
\toprule
degree & name & $R(X, K_X)$  & on $\tilde X_1$  &  on $\tilde X_2$\\
\midrule
1 & $a_0$ &$ (x,0)$ & $x$ &  \\ 
1 & $a_1$ & $(0,u_0)$ &  &  $u_0$\\ 
 2 & $b_0$ &  $(y_0,u_1^2)$ & $y_0$ & $u_1^2$\\
 2 & $b_1$ & $(0,u_0u_1)$  &  & $u_0u_1$\\
 3 & $ c$ & $(xy_1,0) $  & $ xy_1 $  &    \\ 
4 & $ d$ & $(y_1^2,v)$& $ y_1^2$& $ v$  \\
5 & $e$ &$(z,u_1(v-u_1^4))$& $z$& $u_1(v-u_1^4)$  \\
6 & $f$ & $(y_1(y_1^2-y_0^2),w)$& $y_1(y_1^2-y_0^2)$& $w$ \\
7 & $g$ & $(y_1z,u_1w)$& $y_1z $& $u_1w $\\
\bottomrule
 \end{tabular}
\end{table}

\begin{lem}\label{lem: X_1}
The generators in Table \ref{tab: gens DE} induce an immersion
\[ \tilde X_1 \to X_1 \subset \pp(1_{a_0},2_{b_0},3_c,4_d,5_e,6_f,7_g)\]
which maps both points $s_i$ to $(0:1:0:0:0:0:0)$ and is an embedding elsewhere (compare Figure \ref{fig: component DE}). 
The ideal of $X_1$ is  generated by 
\[
\rank\begin{pmatrix} a_0 & d-b_0^2 & e & c & f & g \\ c & f & g & a_0d & (d-b_0^2)d & ed \end{pmatrix}\le1,\\
\]
\begin{align*}
e^2&=b_0(d-b_0^2)^2+a_0^2A_8+a_0cB_6,\\
eg&=b_0f(d-b_0^2)+a_0cA_8+a_0^2dB_6,\\
g^2&=b_0d(d-b_0^2)^2+a_0^2dA_8+a_0cdB_6\ \ (=de^2), 
\end{align*}
where $A_8$, $B_6$ are certain polynomials in $a_0,b_0,c,d,f$.
\end{lem}
\begin{proof}
This can be checked with the computer, but also by hand: the determinantal relations encode the image of 
\[
\begin{tikzcd}[column sep = large]
 \IP(1,2,2,5) \arrow{rrrr}{(x: y_0: xy_1: y_1^2: z : y_1(y_1^2 - y_0^2) : y_1z)} &&&& \IP(1,2,3,4,5,6,7)
\end{tikzcd}
\]
and the polynomials $A_8$ and $B_6$ are determined from \eqref{eq: X_1 tilde} by 
\[ e^2 = z^2 = y_0(y_1+y_0)^2(y_1-y_0)^2+x^2f_8(x,y_0,y_1) = b_0(d-b_0^2)^2+a_0^2A_8+a_0cB_6.\]
The remaining two equations are induced by $g^2 = de^2$ and $eg = y_1 z^2$ using the determinantal relations.
\end{proof}

\begin{lem}\label{lem: X_2}
The generators in Table \ref{tab: gens DE} induce an embedding 
\[ \tilde X_2 \to \IP(1,1,4,6) \to \pp(1_{a_1},2_{b_0},2_{b_1},4_d,5_e,6_f,7_g)\]
identifying $X_2 \isom \tilde X_2$. The ideal  of $X_2$ is generated by 
\[
\rank\begin{pmatrix} a_1 & d-b_0^2 & f & b_1 & e & g \\ b_1 & e & g & a_1b_0 & (d-b_0^2)b_0 & fb_0 \end{pmatrix}\le1
\]
\begin{align*}
f^2&=d(d-b_0^2)^2+a_1C_{11}+b_1D_{10}\\
fg&=de(d-b_0^2)+b_1C_{11}+a_1b_0D_{10}\\
g^2&=b_0d(d-b_0^2)^2+a_1b_0C_{11}+b_0b_1D_{10} \ \ (=b_0f^2)
\end{align*}
where $C_{11}$, $D_{10}$ are general polynomials in $a_1,b_0,b_1,d,e$.
\end{lem}
\begin{proof}
Note that $a_1^2, b_0, b_1, d, f$ generate the even subring of $\IC[u_0, u_1, v, w]$, so the map is an embedding. 

As in the proof of Lemma \ref{lem: X_1}, the determinantal equations cut out the image of $\IP(1,1,4,6)$ and the remaining equations are induced from  \eqref{eq: X_2 tilde} by the equation defining $\tilde X_2$. 
\end{proof}

The canonical model of the glued surface $X$ in $\pp(1,1,2,2,3,4,5,6,7)$ can be computed as the intersection of the two former ideals:
\begin{prop}\label{prop: algebraic DE}
Let $X$ be a $2$-Gorenstein I-surface with every canonical curve
non-reduced and such that the $S_2$-fication of one component is of  type D and the other component is of elliptic type E. 
Then the  canonical model of $X$ is defined by the following 20 equations in $\pp(1_{a_0},1_{a_1},2_{b_0},2_{b_1},3_c,4_d,5_e,6_f,7_g)$:
\begin{equation}\label{eq: ideal DE}
\rank\begin{pmatrix}0&0&a_0&c\\0&0&c&a_0d\\a_1&b_1&d-b_0^2&f\\b_1&a_1b_0&e&g\end{pmatrix}\le2,
\qquad
\begin{cases}e^2&=b_0(d-b_0^2)^2+a_0^2A_8+a_0cB_6\\
eg&=b_0f(d-b_0^2)+a_0cA_8+a_0^2dB_6\\
f^2&=d(d-b_0^2)^2+a_1C_{11}+b_1D_{10}\\
fg&=de(d-b_0^2)+b_1C_{11}+a_1b_0D_{10}\\
g^2&=de^2+a_1b_0C_{11}+b_1^2D_{10}\\
(g^2&=b_0f^2+a_0^2dA_8+a_0cdB_6)
\end{cases}
\end{equation}
If $A,B,C,D$ are general, then $X$ is a stable I-surface, with $\omega_X=\OO_X(1)$ and $\omega_X^{[2]}=\OO_X(2)$ is invertible.
\end{prop}
\begin{proof} 
The intersection of ideals of the components described in Lemma \ref{lem: X_1} and Lemma \ref{lem: X_2} can be calculated by computer or by hand giving \eqref{eq: ideal DE}.
The two components of $X$ are $X_1=X\cap\{a_1=b_1=0\}$ and $X_2=X\cap\{a_0=c=0\}$ and one can confirm that this recovers the ideals of the components.

Now consider a surface $X$ defined by these equations with $A, B, C, D$ general. 
The Hilbert series of $X$ can be determined by computer or follows from Proposition \ref{prop: canonical section} as soon as we have established that $X$ is 2-Gorenstein stable. 

The canonical sheaf $\omega_X = \ko_X(1)$ can be computed as in the proof of Proposition \ref{prop: algebraic type B}.

The sheaf $\OO_\pp(2)$ is invertible on $\pp(1,1,2,2,3,4,5,6,7)$ away from the loci $\pp(3_c,6_f)$, $\pp(4_d)$, $\pp(5_e)$ and $\pp(7_g)$. Clearly, $X$ does not meet the last three of these loci, because monomials $d^3$, $e^2$, $g^2$ appear with non-zero coefficient in some equations defining $X$. Moreover, $X\cap\pp(3,6)$ is also empty, because $c^2$ and $f^2$ also appear in separate equations defining $X$. Hence the restriction $\omega_X^{[2]}=\OO_\pp(2)|_X$ is invertible.

Using the description of the components, one can see that for general equations the resulting surface has slc singularities.
\end{proof}
\begin{rem}
 The non-existence of an involution as in Lemma \ref{lem: DE no involution} can also be checked in this algebraic model. 
\end{rem}


  \def\cprime{$'$}

  \end{document}